\documentclass[reqno]{amsart}
\usepackage[left=1in,right=1in,top=1in,bottom=1in]{geometry}
\usepackage{tikz}
\usetikzlibrary{shapes,decorations,calc,arrows}
\usepackage[foot]{amsaddr}
\usepackage{amsthm}
\usepackage{amscd}
\usepackage{amsfonts}
\usepackage{amsmath}
\usepackage{amssymb}
\usepackage{mathrsfs}
\usepackage{multirow}
\usepackage{verbatim}
\usepackage{url}
\usepackage[hidelinks]{hyperref}
\usepackage{graphicx}
\usepackage{cite}
\usepackage{fancyhdr}
\usepackage{yfonts}
\usepackage{setspace}
\usepackage{titlesec}
\usepackage{enumitem}

\usepackage{ dsfont }

\titleformat{\section}[hang]
{\normalfont\Large\bfseries}
{\thesection.}{0.5em}{}

\titlespacing*{\section}{0pc}{2pc}{0.25pc}

\titleformat{\subsection}[runin]
{\normalfont\large\bfseries}
{\thesubsection}{0.5em}{}

\titlespacing{\subsection}{0pc}{1.5pc}{0.5pc}

\newcommand{\N}{\mathbb{N}}
\newcommand{\Z}{\mathbb{Z}}

\newcommand{\C}{\mathbb{C}}

\definecolor{ggreen}{HTML}{7FDD99}

\definecolor{obnoxious}{HTML}{662288}

\newtheorem{thm}{Theorem}[section]

\newtheorem{lem}[thm]{Lemma}
\newtheorem{cor}[thm]{Corollary}
\newtheorem{innercthm}{Theorem}
\newenvironment{cthm}[1]
  {\renewcommand\theinnercthm{#1}\innercthm}
  {\endinnercthm}

\theoremstyle{definition}
\newtheorem{defi}[thm]{Definition}
\newtheorem{ex}[thm]{Example}

\newtheorem{rem}[thm]{Remark}

\title{\textbf{Schreier's Formula for some Free Probability Invariants}}
\author{Aldo Garcia Guinto}
\address{Department of Mathematics, Michigan State University\hfill \url{garci575@msu.edu}}
\date{}

\begin{document}

\maketitle

\begin{abstract}
Let $G\stackrel{\alpha}{\curvearrowright}(M,\tau)$ be a trace-preserving action of a finite group $G$ on a tracial von Neumann algebra. Suppose that $A \subset M$ is a finitely generated unital $*$-subalgebra which is globally invariant under $\alpha$. We give a formula relating the von Neumann dimension of the space of derivations on $A$ valued on its coarse bimodule to the von Neumann dimension of the space of derivations on $A \rtimes_\alpha G$ valued on its coarse bimodule, which is reminiscent of Schreier's formula for finite index subgroups of free groups. This formula induces a formula for the free Stein dimension (defined by Charlesworth and Nelson) $\dim \text{Der}_c(A,\tau)$ (defined by Shlyakhtenko) and $\Delta$ (defined by Connes and Shlyakhtenko). The latter is done by establishing that $\Delta$ is equal to the von Neumann dimension of a certain subspace of the derivation space of $A$, similar to that of the free Stein dimension, and assuming that $G$ is abelian group. Using the formula for $\Delta$, we recover recent results of Shlyakhtenko on the microstates free entropy dimension.
\end{abstract}

\section*{Introduction}
Let $X = (x_1, \ldots , x_n)$ be a tuple of operators in a tracial von Neumann algebra $(M,\tau)$. Some of the quantitative free probability numerical invariants associated with the distribution of $X$ with respect to $\tau$ are the free entropy dimensions $\delta(X), \delta_0(X), \delta^*(X),$ and $\delta^\star(X)$ (see \cite{Voi94, Voi96, CS05}); the free Fisher information $\Phi^*(X)$ (see \cite{Voi98}); etc. The following free probability numerical invariants are the invariants that we study in this paper. In \cite{CN22}, Charlesworth and Nelson defined an invariant called the free Stein dimension $\sigma(X, \tau)$, which comes from taking a von Neumann dimension of a particular subspace of derivations (see Section 1.3). In \cite{CS05}, Connes and Shlyakhtenko defined an invariant denoted by $\Delta(X,\tau)$, which comes from taking a von Neumann dimension of a closed subspace of Hilbert-Schmidt operators. Theses quantities are invariants of the $*$-algebra generated by $X$, say $A =\C\langle X \rangle$, and we denote these quantities by $\sigma(A,\tau)$ and $\Delta(A,\tau)$. Similar to the free Stein dimension, $\Delta$ can be computed by taking the von Neumann dimension of a certain subspace of derivation on $A$ valued on $L^2(A \otimes A^\circ, \tau \otimes \tau^\circ)$ (see Lemma \ref{Lem:CS_Correp_Der}). In \cite[Theorem 1.1]{MSY20} the authors showed that $\Delta$ being maximal ($\Delta(A,\tau) = n$) implies that there is an absence of rational relations, that is any nontrivial rational function evaluates to an affiliated rational operator of the tuple. In this paper, we also investigate a certain subspace $\text{Der}_c(A,\tau)$ of the derivation space defined by Shlyakhtenko in \cite{Shl09}.

From group theory, the Nielsen--Schreier theorem states that any subgroup $H$ of a free group $F$ is a free group and if $F \cong \mathbb{F}_n$ for $n \in \N$ and $[F: H] < \infty$, then $H \cong \mathbb{F}_k$, where 
 $$k = 1 + [F: H](n-1).$$ 
The formula above is called Schreier's formula. Since there exists a generating set $X$ for $L(\mathbb{F}_n)$ such that the non-microstate free entropy dimension $\delta^\star(X) =n$ (see \cite{Voi98, CS05}), one expects that given a finite index of subfactors $M_0 \subset M_1$, and a finite generating set $S_0$ of $M_0$, there exists a generating set $S_1$ for $M_1$ such that 
 $$\delta^\star(S_1) -1 = [ M_1 :M_0]^{-1} ( \delta^\star(S_0) -1). $$
In \cite{Shl22}, Shlyakhtenko shows that for subfactors of the form $M_0 \subset M_1 = M_0 \rtimes G$ with $G$ a finite abelian group, one has for a given $\varepsilon>0$ the existence of a generating sets $S_0$ for $M_0$ and $S_1$ for $M_1$ for which 
 $$ \delta^*(S_1)-1 \leq [M_1 : M_0 ]^{-1} (\delta^*(S_0) -1) + \varepsilon.$$
This is motivated by a conjecture that has the potential to resolve the free group factor isomorphism problem (see \cite{Shl22} Conjecture 1 and discussion following it). We remark that Example \ref{Ex:Fin_Gen_gp_Delta} is what one would expect if Shlyakhtenko's conjecture has a positive resolution. This paper motivated us to consider how $\sigma$, $\Delta$ behave under such crossed products. Consider a $*$-algebra $A$ that is globally invariant under $\alpha$, we define $A \rtimes_\alpha G$ to be the $*$-algebra generated by $A$ and the unitaries $\{ u_g : g\in G\}$ implementing $\alpha$ (see Section 1.2). We were able to obtain an analog of Schreier's formula for the quantities above with $A$ and $A \rtimes_\alpha G$:

\begin{cthm}A[{Theorem \ref{Thm:Free_Stein_Dim}}] \label{Thm:A}
Let $G\stackrel{\alpha}{\curvearrowright}(M,\tau)$ be a trace-preserving action of a finite group $G$ on a tracial von Neumann algebra and let $A\subset M$ be a finitely generated unital $*$-subalgebra which is globally invariant under $\alpha$. Then 
 $$\sigma( \C[G] \subset A \rtimes_\alpha G,\tau) = \frac{1}{|G|}\sigma(A,\tau).$$
Furthermore, we have
 $$\sigma(A \rtimes_\alpha G,\tau) -1 = \frac{1}{|G|}(\sigma(A,\tau)-1).$$
\end{cthm}

\begin{cthm}B[{Theorem \ref{Thm:Anot_subs}}]\label{Thm:B}
Let $G\stackrel{\alpha}{\curvearrowright}(M,\tau)$ be a trace-preserving action of a finite group $G$ on a tracial von Neumann algebra and let $A\subset M$ be a finitely generated unital $*$-subalgebra which is globally invariant under $\alpha$. Then
 $$\dim \overline{\emph{Der}_c( \C[G] \subset A \rtimes_\alpha G,\tau )}_{((A\rtimes_\alpha G) \otimes (A\rtimes_\alpha G)^\circ)''} = \frac{1}{|G|}\dim\overline{\emph{Der}_c(A,\tau)}_{(A \otimes A^\circ)''}.$$
Furthermore, we have
 $$\dim \overline{\emph{Der}_c(A \rtimes_\alpha G, \tau)}_{((A\rtimes_\alpha G) \otimes (A\rtimes_\alpha G)^\circ)''} -1 = \frac{1}{|G|}(\dim \overline{\emph{Der}_c(A,\tau)}_{(A \otimes A^\circ)''}-1).$$
\end{cthm} 

\begin{cthm}C[{Theorem \ref{Thm:Delta_Dim}}] \label{Thm:C}
Let $G\stackrel{\alpha}{\curvearrowright}(M,\tau)$ be a trace-preserving action of a finite abelian group $G$ on a tracial von Neumann algebra and let $A\subset M$ be a finitely generated unital $*$-subalgebra which is globally invariant under $\alpha$. Then 
 $$\Delta( \C[G] \subset A \rtimes_\alpha G,\tau) = \frac{1}{|G|}\Delta(A,\tau).$$
Furthermore, we have
 $$\Delta(A \rtimes_\alpha G,\tau) -1 = \frac{1}{|G|}(\Delta(A,\tau)-1).$$
\end{cthm}

To prove the above theorems, we first decompose the derivation space of $A \rtimes_\alpha G$ that vanishes on $\C[G]$ into a finite direct sum of twisted copies of the derivation space of $A$. Since each derivation on $A \rtimes_\alpha G$ can be decomposed into a derivation on $\C[G]$ and a derivation on $A \rtimes_\alpha G$ that vanishes on $\C[G]$, this gives a formula relating the von Neumann dimensions of the derivation space of $A$ and von Neumann dimension of the derivation space of $A \rtimes_\alpha G$. Furthermore, we show that both decompositions can be restricted to the subspace $\text{Der}_c(A,\tau)$ and the subspaces corresponding to $\sigma$ and $\Delta$, and that the derivations on $\C[G]$ are equal to the closure of the restricted subspace on $\C[G]$. This yields the formulas in the above theorems. Using the known inequality, $\delta_0 \leq \Delta$ (see \cite[Corollary 4.6]{CS05}), we can obtain a sharper bound on the microstates free entropy dimension which is independent of the choice set of generators on $A \rtimes_\alpha G$.

In Section 1 after some preliminaries, we state a generalized version of \cite[Theorem 2.1]{CN22}(see Lemma \ref{Lem:Decomp}), which uses the same proof. In Section 2 we show how one can extend a derivation on $A$ to a derivation on $A \rtimes_\alpha G$ that vanishes on $\C[G]$ and restrict a derivation on $A \rtimes_\alpha G$ that vanishes on $\C[G]$ to a derivation on $A$. From this we decompose the derivation space of $A \rtimes_\alpha G$ that vanishes on $\C[G]$ to a finite direct sum of the derivation space of $A$ as $(A \otimes A^\circ)''$-modules. In Section 3 we state and prove Theorem A. In Section 4 we state and prove Theorem B. In Section 5 we state and prove Theorem C. We relate $\Delta$ with a subspace of $\text{Der}(A,\tau)$ coming from defining a locally convex structure to $\text{Der}(A,\tau)$. We recover results in \cite{Shl22} and obtain a sharper bound on microstates free entropy dimension $\delta_0$.

\section*{Acknowledgement} I would like to thank my advisor, Prof. Brent Nelson for the initial idea of this paper, and many helpful suggestions and discussions. I would like to thank Rolando de Santiago and Krishnendu Khan for pointing out Corollary 2.8. I would like to thank Srivatsav Kunnawalkam Elayavalli for referring me to the result of Higman, Neumann and Neumann.

\section{Preliminaries}

\subsection{Notation} 
Throughout, a tracial von Neumann algebra is a pair $(M,\tau)$ consisting of a finite von Neumann algebra $M$ with a choice of a faithful normal tracial state $\tau$. We denote by $L^2(M,\tau)$ the GNS Hilbert space corresponding to $\tau$ and identify $M$ with its representation on this space. We let $M^\circ = \{ x^\circ: x \in M\}$ denote the opposite von Neumann algebra, represented on $L^2(M^\circ, \tau^\circ)$ which can be identified with the conjugate Hilbert space of $L^2(M,\tau)$. We let $M\bar\otimes M^\circ$ denote the von Neumann algebra tensor product, which is equipped with the tensor product trace $\tau \otimes \tau^\circ$ and represented on $L^2(M\bar\otimes M^\circ, \tau \otimes \tau^\circ)$. 

Let $G\stackrel{\alpha}{\curvearrowright}M$ be a trace-preserving action of a finite group $G$. Recall that the crossed product $M \rtimes_\alpha G$ is the von Neumann algebra generated by $M$ and $\C[G]= \text{span}\{u_g: g\in G\}$ with the relations $\alpha_g(m)=u_g mu_g^*$ for $m \in M$. For each $b \in M \rtimes_\alpha G$, we have the unique representation 
 $$b = \sum_{g \in G}a_g u_g,$$
where $a_g \in M$, and the trace on $M$ extends to a trace on $M \rtimes_\alpha G$ via
 $$\tau(b)= \tau\left(\sum_{g \in G}a_g u_g\right) = \tau (a_e).$$
It follows that $L^2(M \rtimes_\alpha G,\tau) = \bigoplus_{g \in G} L^2(M,\tau)u_g$. We denote by $J_\tau$ and $J_{\tau \otimes \tau^\circ}$ the Tomita conjugation operators on $L^2(M \rtimes_\alpha G,\tau)$ and $L^2((M \rtimes_\alpha G)\otimes (M \rtimes_\alpha G)^\circ, \tau \otimes \tau^\circ)$ respectively. These are determined by $J_\tau(x) =x^*$ and $J_{\tau \otimes \tau^\circ}(a\otimes b)=a^*\otimes b^*$ for $x,a,b \in M \rtimes_\alpha G$. For $S \subset L^2(M \rtimes_\alpha G, \tau)$, we denote by $[S]$ the projection onto the closed span of $S$.  

Additionally, we have that $L^2(M \bar{\otimes} M^\circ,\tau \otimes \tau^\circ)$ is a $M$-$M$ bimodule with the actions define as follows
\begin{equation}\label{Eq:Right_Left_Action}
    x \cdot (a \otimes b^\circ) \cdot y = (xa) \otimes (by)^\circ= (x \otimes y^\circ) ( a \otimes b^\circ),
\end{equation} 
where $a,b,x,y \in M$. Notice that these actions commute with the right action implemented by $M \bar{\otimes}M^\circ.$ For a unital $*$-subalgebra $B \subset M$, we let $L^2_B(M\otimes M^\circ,\tau \otimes \tau^\circ)$ denote the subspace of $B$-central vectors.

\subsection{Finitely Generated $*$-Algebra} Our primary interest will be finitely generated unital $*$-subalgebras $A \subset M$ that are globally invariant under $\alpha$. That is, there exists a finite self adjoint subset $X \subset M$ such that $A = \mathbb{C}\langle X\rangle$, where $\mathbb{C}\langle X\rangle$ is the unital $*$-algebra generated by $X$, and $\alpha_g (A)\subset A$ for all $g \in G$. Hence, if we define $A \rtimes_\alpha G : = \mathbb{C} \langle A, \{u_g: g \in G\}\rangle$, then every $b \in A \rtimes_\alpha G$ has the form $\sum_{g \in G}a_gu_g$ where $a_g \in A$. The von Neumann algebras generated by $A\rtimes_\alpha G$ and $(A\rtimes_\alpha G )\otimes (A\rtimes_\alpha G)^\circ$ will be denoted by $(A\rtimes_\alpha G)''$ and $((A\rtimes_\alpha G) \otimes (A\rtimes_\alpha G)^\circ)''$ respectively. Their $L^2$-closures will be denoted by $L^2(A\rtimes_\alpha G,\tau) \subset L^2( M\rtimes_\alpha G, \tau)$ and $L^2((A\rtimes_\alpha G) \otimes (A\rtimes_\alpha G)^\circ, \tau \otimes \tau^\circ) \subset L^2( (M\rtimes_\alpha G )\otimes (M\rtimes_\alpha G)^\circ,\tau \otimes \tau^\circ)$ respectively. By the decomposition on $L^2(A \rtimes_\alpha G)$ we have, 
 $$L^2( (A \rtimes_\alpha G) \otimes (A \rtimes_\alpha G)^\circ, \tau \otimes \tau^\circ)  = \bigoplus_{g,h \in G} L^2(A) u_g \otimes (L^2(A)u_h)^\circ= \bigoplus_{g,h \in G} L^2(A \otimes A^\circ) (u_g \otimes u_h^\circ).$$

\begin{rem}\label{Rem:Globally}
We note that we can always restrict to $A$ being globally invariant. Let $A \subset M$ be a finitely generated unital $*$-subalgebra. We claim that $\bigvee_{g \in G}\alpha_g(A)$ is a finitely generated unital $*$-subalgebra that is globally invariant, where $\bigvee_{g \in G}\alpha_g(A)$ is the $*$-algebra generated by $\alpha_g(A)$ for all $g \in G$. Since $G$ is finite, it suffices to show that for $A \subset M$ a finitely generated unital $*$-subalgebra, one has that $\alpha_g(A)$ is also a finitely generated unital $*$-subalgebra for all $g \in G$. This is true, since $\alpha_g(\C\langle X \rangle) = \C \langle \alpha_g(X)\rangle$. In particular, if $A$ is globally invariant with $A = \C \langle X\rangle$, then $\C\langle\alpha_g(X)\rangle = A$.
\end{rem}

\subsection{Derivation Spaces} Let $A \subset M$ be a finitely generated unital $*$-subalgebra. We denote by $\text{Der}(A,\tau)$ the vector space of all derivations $d: A \to L^2(A \otimes A^\circ, \tau \otimes \tau^\circ)$.  Consider the following subspaces:
\begin{enumerate}
    \item[a)] $\text{Der}_{1 \otimes 1}(A,\tau):= \{ d \in \text{Der}(A,\tau):  1 \otimes 1^\circ \in \text{dom}(d^*)\},$ where $d$ is viewed as a densely defined operator \\ $L^2(A,\tau) \to L^2(A \otimes A^\circ, \tau \otimes \tau^\circ)$,
    \item[b)] $\text{Der}(B \subset A, \tau):= \{ d \in \text{Der}(A,\tau): d|_B \equiv 0\}$,  where $B$ is a unital $*$-subalgebra of $A$;
    \item[c] $\text{InnDer}(A,\tau):= \{ d \in \text{Der}(A,\tau): d =[ \,\cdot\, ,\xi]: \xi \in L^2(A\otimes A^\circ, \tau\otimes \tau^\circ)\}$, where $ [\,\,,\,]$ is the commutator.
\end{enumerate}
Note that for $d \in \text{Der}_{1 \otimes 1}(A,\tau)$, the condition $1 \otimes 1^\circ \in \text{dom}(d^*)$ implies $A \otimes A^\circ \subset \text{dom}(d^*)$, by the proof in Proposition 4.1 in\cite{Voi98}. Since $A \otimes A^\circ$ is dense in $L^2(A \otimes A^\circ,\tau \otimes \tau^\circ)$, one has that $d^*$ is densely defined. Thus, $d$ is a closable derivation. 

The right $(A \otimes A^\circ)''$-action on $L^2(A \otimes A^\circ, \tau \otimes \tau^\circ)$, coming from (\ref{Eq:Right_Left_Action}), induces a right $(A \otimes A^\circ)''$-action on $\text{Der}(A,\tau)$ defined as
 $$ (d \cdot m) (x)= d(x) m,$$
where $d \in \text{Der}(A,\tau)$ and $m \in (A \otimes A^\circ)''$. Thus $\text{Der}(A,\tau)$ is a $(A \otimes A^\circ)''$-module. We have that both $\text{Der}(B \subset A, \tau)$ and $\text{InnDer}(A,\tau)$ are $(A \otimes A^\circ)''$-submodules, while $\text{Der}_{1 \otimes 1}(A,\tau)$ is only $(A \otimes A^\circ)$-submodule. Then we have that $\overline{\text{Der}_{1 \otimes 1} (A,\tau)}$ is $(A \otimes A)''$-submodule.

For any closed $(A \otimes A^\circ)''$-invariant subspace $\mathcal{H} \leq \text{Der}( B \subset A ,\tau)$, $\phi_X(\mathcal{H})$ is a right Hilbert $(A\otimes A^\circ)''$-submodule of $L^2(A \otimes A^\circ,\tau \otimes \tau^\circ)^X$, since $\phi_X$ is a right $(A \otimes A^\circ)''$-linear. So one can compute its von Neumann dimension and by Lemma 1.2 in \cite{CN22}, this is independent of the generating set $X$. In \cite{CN22}, Charlesworth and Nelson defined the \textit{free Stein dimension} of $A$ with respect to $\tau$ as 
 $$ \sigma(A,\tau) := \dim\overline{\text{Der}_{1 \otimes 1}(A,\tau)}_{(A\otimes A^\circ)''}.$$ 
They defined the following map 
    \begin{align*}
        \phi_X: \text{Der}(A,\tau) &\to L^2(A \otimes A^\circ,\tau\otimes \tau^\circ)^X, \\
        d &\mapsto (d(x))_{x \in X},
    \end{align*}
where $A = \C\langle X \rangle$. They showed that the map was an injective, right $(A \otimes A^\circ)''$-linear and the image is closed in $L^2(A \otimes A^\circ,\tau\otimes \tau^\circ)^X$ (see \cite[Lemma 1.1]{CN22}). This allows us to define an inner product on $\text{Der}(A,\tau)$, given by
 $$ \langle d_1, d_2 \rangle_X := \sum_{x \in X} \langle d_1(x), d_2(x)\rangle.$$
Then $\text{Der}(A,\tau)$ with the above inner product is a Hilbert space. The topology coming from this inner product is simply pointwise convergence. Notice that the inner product depends on the choice of generators, that is, if one changes the generating set, then one also changes the inner product. However, for two distinct generating sets $X$ and $X'$ of $A$, we have that $\| \cdot\|_X$ is norm equivalent to $\|\cdot\|_{X'}$. 

The following lemma is a generalization of Theorem 2.1 in \cite{CN22}.

\begin{lem}\label{Lem:Decomp} Let $\mathcal{H}$ be a ${(A \otimes A^\circ)''}$-invariant subspace of $\emph{Der}(A,\tau)$ containing $\emph{InnDer}(A,\tau)$. For a finite dimensional $B \subset A$, we have 
 $$ \dim \overline{\mathcal{H}}_{(A \otimes A^\circ)''} = \dim \emph{Der}(B ,\tau)_{(B \otimes B^\circ)''} + \dim [ \overline{\mathcal{H} \cap \emph{Der}(B\subset A,\tau)}]_{(A \otimes A^\circ)''}.$$
\end{lem}
\begin{proof}
Fix a finite self adjoint subset $X\subset A$ satisfying $A = \C\langle X \rangle$. Since $B\subset A$ is finite dimensional, we have
 $$ (B ,\tau) = \left( \sum^d_{i=1}M_{n_i}(\C), \sum^d_{i=1}\alpha_i\text{tr}_{n_i} \right)$$
for some $n_1, \ldots , n_d \in \N$ and $\alpha_1, \ldots , \alpha_d >0$ satisfying $\sum^d_{i=1}\alpha_i=1$. Let 
 $$ E := \{e_{j,k}^{(i)} : 1 \leq i \leq d, 1 \leq j,k \leq n_i \}$$
be a family of multi-matrix units for $B$. Then 
$$ p : = \sum^d_{i=1} \frac{1}{n_i} \sum^{n_i}_{j,k=1} e^{(i)}_{j,k} \otimes (e^{(i)}_{k,j})^\circ$$
is the projection from $L^2(A \otimes A^\circ, \tau \otimes \tau^\circ)$ onto the $B$-central vectors $L^2_B(A \otimes A^\circ, \tau \otimes \tau^\circ).$

Now, set $Y:= \{ exe': x \in X, e,e' \in E\} \subset A$. Since $1 \in \text{span}(E)$, one has that $\C\langle Y \rangle =A$. For $d \in \mathcal{H}\cap\text{Der}(B\subset A,\tau)$,
    \begin{align*}
        \sum_{y \in Y}y^*d(y) &= \sum_{\substack{x \in X \\ e,e'\in E}} ((e')^*x^*e^*)\cdot d(exe')\\
        &= \sum_{\substack{x \in X \\ e,e'\in E}} ((e')^*x^*e^*e ) \cdot d(x)\cdot e'\\
        &= \sum_{e'\in E} ((e')^*\otimes (e')^\circ )\sum_{\substack{x \in X \\ e\in E}} (x^*e^*e)\cdot d(x) \\
        &= \sum_{i =1}^d\sum_{j,k =1}^{n_i} (e^{(i)}_{j,k}\otimes (e^{(i)}_{k,j})^\circ )\sum_{\substack{x \in X \\ e\in E}} (x^*e^*e)\cdot d(x).
    \end{align*}
Thus $\sum_{y \in Y} y^*d(y)$ is $B$-central, when $d \in \mathcal{H}\cap\text{Der}(B\subset A,\tau)$. Similarly $\sum_{y \in Y} d(y)y^*$ is $B$-central. For any $\xi\in L^2(A \otimes A^\circ, \tau \otimes \tau^\circ)$ and $d \in \mathcal{H}\cap\text{Der}(B\subset A,\tau)$ we have,
    \begin{align*}
        \langle [\, \cdot \, , \xi], d \rangle_Y&= \sum_{y \in Y} \langle [y,\xi], d(y) \rangle\\
        &=\sum_{y \in Y} \langle \xi,[y^*,  d(y)] \rangle\\
        &=\sum_{y \in Y} \langle p \xi,[y^*,  d(y)] \rangle\\
        &= \langle [\, \cdot \,, p\xi], d \rangle_Y.
    \end{align*}
Hence, $[ \, \cdot \,, p\xi]$ is the projection of $[\,\cdot\, , \xi]$ onto $\mathcal{H}\cap\text{Der}(B\subset A,\tau)$. 

Since $B$ is finite dimensional, for each $d \in \mathcal{H}$ there exists $\xi \in L^2(A \otimes A^\circ ,\tau \otimes \tau^\circ)$ with $d|_{B} = [ \, \cdot\, , \xi] = [\, \cdot\, , (1-p)\xi]|_B$ (see \cite[Theorem 2.2]{Pet09}). Then 
 $$d = (d-[ \,\cdot\,, (1-p)\xi]) + [ \, \cdot \,, (1-p)\xi]$$
is an orthogonal decomposition with respect to $\langle \cdot, \cdot \rangle_Y$. Since $\text{InnDer}(A,\tau) \subset \mathcal{H}$, the first term is in $\mathcal{H}$. It follows that 
 $$ \mathcal{H} = \mathcal{H} \cap \text{Der}(B \subset A, \tau)\oplus \{ [\, \cdot \, , \xi] : \xi \in L^2_B(A \otimes A^\circ,\tau \otimes \tau^\circ)^\perp\}.$$
Since $L^2_B(A \otimes A^\circ, \tau \otimes \tau^\circ)^\perp \leq L^2_A(A \otimes A^\circ, \tau \otimes \tau^\circ)^\perp$, taking the von Neumann dimensions of their closure in the above orthogonal decomposition and using Lemma 1.4 in \cite{CN22}, one has
 $$\dim \overline{\mathcal{H}}_{(A \otimes A^\circ)''} = \dim [ \overline{\mathcal{H} \cap \text{Der}(B\subset A,\tau)}]_{(A \otimes A^\circ)''} + \dim L^2_B(A \otimes A^\circ, \tau \otimes \tau^\circ)^\perp_{(A \otimes A^\circ)''}.$$
Lastly, we have
    \begin{equation*}
        \dim L^2_B(A \otimes A^\circ, \tau \otimes \tau^\circ)^\perp_{(A \otimes A^\circ)''}= \tau \otimes \tau^\circ (1-p)= 1 - \sum^d_{i=1}\frac{\alpha^2_i}{n^2_i} =\dim \text{Der}(B ,\tau)_{(A \otimes A^\circ)''}. \qedhere
    \end{equation*}
\end{proof}

\section{Decomposing Derivations}    
In this section, we will show that the derivation space of $A \rtimes_\alpha G$ that vanishes on $\C[G]$ can be decomposed to a finite direct sum of the derivation space of $A$. To do so, we need to consider how to extend a derivation on $A$ to a derivation on $A \rtimes_\alpha G$, and how to restrict a derivation on $A \rtimes_\alpha G$ to a derivation on $A$. First, set $p_{g,h} := [ L^2(A\otimes A^\circ,\tau \otimes \tau^{\circ}) (u_g \otimes u_h^\circ)]$ for $g,h \in G$, which gives us the pairwise orthogonal family of projections $\{p_{h,g}\}_{g,h \in G}$. Since for each $g,h \in G$, $L^2(A\otimes A^\circ,\tau \otimes \tau^{\circ}) (u_g \otimes u_h^\circ)$ is invariant under $A \otimes A^\circ$, it follows that $p_{g,h} \in (A\otimes A^\circ)'$ for all $g,h \in G$.

\begin{lem}\label{Lem:rel_of_p_gh} Let $p_{g,h}$ be as above. For $g, h ,k, \ell \in G$:
    \begin{enumerate}
        \item[(1)] $ p_{g,h} (u_k \otimes u_\ell^\circ) = (u_k \otimes u_\ell^\circ) p_{k^{-1}g, h\ell^{-1}}$,
        \item[(2)] $J_{\tau \otimes \tau^\circ}p_{g,h} = p_{g^{-1},h^{-1}}J_{\tau \otimes \tau^\circ}$.
    \end{enumerate}
\end{lem}
\begin{proof} 
    \begin{enumerate}
        \item[]
        \item[(1)] Let $s,r \in G$ and $x,y \in A$, then 
            \begin{align*}
                p_{g,h} (u_k \otimes u_\ell^\circ)(xu_s \otimes (y u_r)^\circ)&=p_{g,h}(\alpha_k(x) u_{ks} \otimes (yu_{ r \ell})^\circ)\\
                &=\delta_{g =ks}\delta_{h=r\ell} \alpha_k(x) u_{ks} \otimes (yu_{ r \ell})^\circ \\
                &=(u_k \otimes u_\ell^\circ)\delta_{g =ks}\delta_{h=r\ell}(x u_s \otimes (y u_r )^\circ)\\
                &=(u_k \otimes u_\ell^\circ)p_{k^{-1}g, h\ell^{-1}}(x u_s \otimes (y u_r )^\circ).
            \end{align*}
        \item[(2)] Let $s,r \in G$ and $x,y \in A$, then 
            \begin{align*}
                p_{g,h} J_{\tau\otimes \tau^\circ}(xu_s \otimes (y u_r)^\circ)&=p_{g,h}(\alpha_{s^{-1}}(x^*) u_{s}^* \otimes (\alpha_r(y)^*u_{r}^*)^\circ)\\
                &=\delta_{g =s^{-1}}\delta_{h=r^{-1}} (\alpha_{s^{-1}}(x^*) u_{s}^* \otimes (\alpha_r(y)^*u_{r}^*)^\circ)\\
                &=J_{\tau \otimes \tau^\circ}\delta_{g =s^{-1}}\delta_{h=r^{-1}}(x u_s \otimes (y u_r )^\circ)\\
                &=J_{\tau \otimes \tau^\circ}p_{g^{-1}, h^{-1}}(x u_s \otimes (y u_r )^\circ).\qedhere
            \end{align*}
    \end{enumerate} 
\end{proof}

The following lemma characterizes derivations on $A\rtimes_\alpha G$ that vanish on $\C[G]$ as derivations that satisfy a type of covariant condition.    

\begin{lem}\label{Lem:fin_dim_and_cov} Let $(M,\tau)$ be a tracial von Neumann algebra with a finitely generated unital $*$-subalgebra $A \subset M$. Let $\alpha$ be a trace-preserving action of a finite group $G$ on $(M, \tau)$ such that $A$ is globally invariant under $\alpha$. Suppose $D \in \emph{Der}(A\rtimes_\alpha G,\tau)$. Then $ D \in\emph{Der}(\C[G] \subset A \rtimes_\alpha G, \tau)$ if and only if $D$ satisfies $D(u_g b u_g^*) = u_g \cdot D(b) \cdot u_g^*$ for all $ g\in G$, $b \in A \rtimes_\alpha G$.
\end{lem}
\begin{proof}
Suppose that $D \in \text{Der}(\C[G] \subset A \rtimes_\alpha G,\tau)$. Then by the Leibniz rule, we have the conclusion. Conversely, suppose $D$ satisfies the covariant condition above. Since $\C[G]$ is a finite dimensional, we have $D|_{\C[G]}$ is bounded and hence inner (see \cite[Theorem 2.2]{Pet09}), say implemented by $\xi \in L^2((A \rtimes_\alpha G)\otimes (A \rtimes_\alpha G)^\circ, \tau \otimes \tau^\circ)$. Then for $x \in \C[G]$ and $g \in G$,  
    \begin{align*}
        [x, \xi - u_g^*\xi u_g] &= [x, \xi] - [x, u_g^*\xi u_g]\\
        &=[x, \xi] - xu_g^* \xi u_g- u_g^*\xi u_gx\\
        &=[x,\xi] - u_g^*(u_gxu_g^* \xi - \xi u_gxu_g^*)u_g \\
        &= [x, \xi] - u_g^*[u_g x u_g^*, \xi] u_g\\
        &= 0.
    \end{align*}
The last equality comes from $D|_{\C[G]}$ being covariant. It follows that $\xi - u_g^*\xi u_g \in L^2_{\C[G]}((A \rtimes_\alpha G )\otimes (A \rtimes_\alpha G)^\circ, \tau \otimes \tau^\circ)$ for all $g \in G.$ Since
 $$\xi - \frac{1}{|G|} \sum_{g \in G}u_g^*\xi u_g =\frac{1}{|G|} \sum_{g\in G}(\xi -u_g^*\xi u_g) \in L^2_{\C[G]}((A \rtimes_\alpha G )\otimes (A \rtimes_\alpha G)^\circ, \tau \otimes \tau^\circ)$$
and $\sum u_g\xi u_g^* \in L^2_{\C[G]}((A \rtimes_\alpha G )\otimes (A \rtimes_\alpha G)^\circ, \tau \otimes \tau^\circ)$, one has $\xi \in L^2_{\C[G]}((A \rtimes_\alpha G)\otimes (A \rtimes_\alpha G)^\circ, \tau \otimes \tau^\circ).$ Hence, $D|_{\C[G]}=0 $.
\end{proof}

For $d \in \text{Der}(A,\tau)$, we view $d$ as valued in $L^2(A \otimes A^\circ,\tau \otimes \tau^\circ) (u_e \otimes u_e^\circ) \leq L^2((A \rtimes_\alpha G) \otimes (A \rtimes_\alpha G)^\circ,\tau \otimes \tau^\circ)$. One may hope to extend $d$ simply by setting $d(u_g)=0$ for all $g \in G$. The issue is that this extension may fail to be a derivation  on $A \rtimes_\alpha G$. Indeed, if it were, then by the Leibniz rule we would get
 $$ d( \alpha_g ( x) ) = d(u_gxu_g^*) = u_gd(x)u_g^*,$$
for $x \in A$, but this is never true since $d(a_g(x)), d(x) \in L^2(A \otimes A^\circ ,\tau \otimes \tau^\circ)(u_e \otimes u_e^\circ).$ We correct this by averaging over $u_g^* d(\alpha_g(x))u_g$.

\begin{lem}\label{Lem:Der_A_to_B}
For $d \in \emph{Der}(A,\tau)$ and $h \in G$, we define $d^h : A \rtimes_\alpha G \to L^2((A \rtimes_\alpha G) \otimes (A \rtimes_\alpha G)^\circ,\tau\otimes \tau^\circ)$ by
 $$ d^h \left(\sum_{k\in G} a_k u_k\right):= J_{\tau \otimes \tau^\circ}(u_e \otimes (u^*_h)^\circ)J_{\tau \otimes \tau^\circ} \sum_{k \in G}\sum_{g \in G} u_g^* \cdot d(\alpha_g(a_k))\cdot (u_gu_k),$$
where $a_k \in A$ for all $k \in G$. Then $d^h \in \emph{Der}(\C[G] \subset A \rtimes_\alpha G, \tau)$. Furthermore, $\{d^h\}_{h \in G}$ are orthogonal with respect to the inner product coming from any generating set $Y$ of $A \rtimes_\alpha G$ with the form $Y = X \cup \{u_g: g\in G\}$, where $A = \C\langle X\rangle$.
\end{lem}
\begin{proof}
First we will show that $d^h$ is a derivation on $A \rtimes_\alpha G$ for $d \in \text{Der}(A,\tau)$ and $h \in G$. Using linearity of $d^h$, we can take $x,y \in A \rtimes_\alpha G$ such that $x =  au_r$ and $y = b u_s$, where $a,b \in A$ and $r,s \in G$. So, 
 $$d^h(xy)=  d^h(a \alpha_r(b) u_{rs})=J_{\tau \otimes \tau^\circ}(u_e \otimes (u^*_h)^\circ)J_{\tau \otimes \tau^\circ}\sum_{g \in G}  (u_g^* \otimes (u_{grs})^\circ)d(\alpha_g(a) \alpha_{gr}(b) ).$$
Then,
    \begin{align*}
        d^h(xy)&=J_{\tau \otimes \tau^\circ}(u_e \otimes (u^*_h)^\circ)J_{\tau \otimes \tau^\circ} \sum_{g \in G}  (u_g^* \otimes (u_{grs})^\circ)[\alpha_g(a) \cdot d(\alpha_{gr}(b) ) +d(\alpha_g(a) ) \cdot \alpha_{gr}(b) ] \\
        &=J_{\tau \otimes \tau^\circ}(u_e \otimes (u^*_h)^\circ)J_{\tau \otimes \tau^\circ} \sum_{g \in G}  [ (au_g^* \otimes (u_{grs})^\circ) d(\alpha_{gr}(b) ) +(u_g^* \otimes (u_{gr} b u_s)^\circ )d(\alpha_g(a) ) ] \\
        &=J_{\tau \otimes \tau^\circ}(u_e \otimes (u^*_h)^\circ)J_{\tau \otimes \tau^\circ} \sum_{g \in G}  [ (au_{gr^{-1}}^* \otimes (u_{gs})^\circ) d(\alpha_{g}(b))] +\sum_{g \in G}[(u_g^* \otimes (u_{gr} b u_s)^\circ )d(\alpha_g(a) ) ] \\
        &=J_{\tau \otimes \tau^\circ}(u_e \otimes (u^*_h)^\circ)J_{\tau \otimes \tau^\circ} \sum_{g \in G}  [ (xu_{g}^* \otimes (u_{gs})^\circ) d(\alpha_{g}(b) ) +(u_g^* \otimes (u_gu_r y)^\circ )d(\alpha_g(a) ) ],
    \end{align*}
and from $J_{\tau \otimes \tau^\circ}(u_e \otimes (u_h^*)^\circ)J_{\tau \otimes \tau^\circ} \in ((A \rtimes_\alpha G)\otimes (A\rtimes_\alpha G)^\circ)'$, we get
    \begin{align*}
        d^h(xy)&= (x \otimes (u_s)^\circ) d^h (b) +  (u_e \otimes (u_r y)^\circ )d^h(a)  \\
        &=x \cdot d^h(b) \cdot u_s + d^h(a) \cdot (u_r y)\\
        &=x \cdot d^h (y) +  d^h(x)\cdot y.
    \end{align*}
The last equality comes from $\delta^h(au_s) = \delta^h(a)\cdot u_s$. Thus $d^h \in \text{Der}( \C[G] \subset A\rtimes_\alpha G,\tau)$, since $d(1)=0$. 

Lastly, we want to show that $\{d^h: h \in G\}$ are orthogonal with respect to the inner product coming from any generating set $Y$ of the form $Y = X \cup \{ u_g :g \in G\}$ satisfying $A= \C \langle X \rangle$. For $h,h' \in G$ with $h \neq h'$, we use that $d^h$ vanishes on $\C[G]$ and $\{p_{g,h}\}_{g,h \in G}$ is a pairwise orthogonal family, one has
    \begin{align*}
        \langle d^h, d^{h'}\rangle_Y&= \sum_{y \in Y} \langle d^h(y), d^{h'}(y)\rangle\\
        &= \sum_{x \in X} \langle d^h(x), d^{h'}(x)\rangle\\
        &= \sum_{x \in X} \sum_{g,k \in G}\langle p_{g^{-1}, hg} d^h(x), p_{k^{-1},h'k}d^{h'}(x)\rangle\\
        &= \sum_{x \in X} \sum_{g \in G}\langle p_{g^{-1}, hg} d^h(x), p_{g^{-1},h'g}d^{h'}(x)\rangle\\
        &=0.\qedhere
    \end{align*} 
\end{proof}

Next, for $D \in \text{Der}(A \rtimes_\alpha G,\tau)$, we are interested in obtaining a derivation on $A$ using $D$. 
Define 
 $$ D_{g,h} := J_{\tau \otimes \tau^\circ}(u_g \otimes u_h^\circ) J_{\tau \otimes \tau^\circ} p_{g,h}D|_A, $$
and if $g =e$, we write $D_h:= D_{e,h}$. By identifying $L^2(A \otimes A^\circ,\tau \otimes \tau^\circ)$ with $L^2(A \otimes A^\circ, \tau \otimes \tau^\circ)( u_e \otimes u_e^\circ)$, we view the map above as valued in $L^2(A \otimes A^\circ, \tau \otimes \tau^\circ)$. The following lemma shows that $D_{g,h}$ is a derivation on $A$.

\begin{lem}\label{Lem:Der_B_to_A} For $D \in \emph{Der}(A \rtimes_\alpha G, \tau) $ and $g,h \in G$, let $D_{g,h}$ be defined as above. Then $D_{g,h} \in \emph{Der}(A,\tau)$. 
\end{lem}
\begin{proof}
For $D \in\text{Der}( A \rtimes_\alpha G, \tau)$ and $a ,b \in A$, we have 
 $$D_{g,h}(ab) =J_{\tau \otimes \tau^\circ} (u_g \otimes u_h^\circ) J_{\tau \otimes \tau^\circ} p_{g,h} D(ab)=J_{\tau \otimes \tau^\circ} (u_g \otimes u_h^\circ) J_{\tau \otimes \tau^\circ} p_{g,h} [ (u_e \otimes (bu_e)^\circ  )D(a)  + (au_e \otimes u_e^\circ)  D(b)],$$
and by $p_{g,h}, J_{\tau \otimes \tau^\circ}(u_g \otimes u_h^\circ)J_{\tau \otimes \tau^\circ}  \in (A\otimes A^\circ)'$, it follows that
    \begin{align*}
        D_{g,h}(ab)&=(u_e \otimes (bu_e)^\circ )J_{\tau \otimes \tau^\circ} (u_g \otimes u_h^\circ )J_{\tau \otimes \tau^\circ} p_{g,h} D(a)  + (au_e \otimes u_e^\circ)  J_{\tau \otimes \tau^\circ} (u_g \otimes u_h^\circ) J_{\tau \otimes \tau^\circ} p_{g,h}D(b)]\\
        &=D_{g,h}(a)\cdot b + a\cdot D_{g,h}(b).
    \end{align*}
Hence, $D_{g,h} \in \text{Der}(A,\tau)$.
\end{proof}

Given $d \in \text{Der}(A,\tau)$ and $m \in (A \otimes A^\circ)''$, one has $(d\cdot m)^h =d^h \cdot (1 \otimes \alpha_{h})(m)$. This occurs since $J_{\tau \otimes \tau^\circ}( u_e \otimes (u_h)^\circ) J_{\tau \otimes \tau^\circ}$ interacts with the right $(A \otimes A^\circ )''$-action. To accommodate this, we define $(\text{Der}(A,\tau))_{1 \otimes \alpha_h}$ to be the same module as $\text{Der}(A,\tau)$ but with a right action by $(A \otimes A^\circ)''$ defined as
 $$d  \cdot_h m = d \cdot (1 \otimes \alpha_h)(m),$$
where $d \in \text{Der}(A,\tau)$, $m \in (A \otimes A^\circ)''$ and the right hand side is the usual $(A \otimes A^\circ)''$-action on $\text{Der}(A,\tau)$. 

\begin{thm}\label{Thm:Bij_Map}
Let $(M,\tau)$ be a tracial von Neumann algebra with a finitely generated unital $*$-subalgebra $A \subset M$. Let $\alpha$ be a trace-preserving action of a finite group $G$ on $(M, \tau)$ such that $A$ is globally invariant under $\alpha$. Then 
    \begin{enumerate}
        \item[(1)] The map $(\emph{Der}(A,\tau))_{1\otimes \alpha_{h^{-1}}} \to \emph{Der}( \C[G] \subset A \rtimes_\alpha G, \tau)$ defined $ d \mapsto d^h$ is an injective, right $(A \otimes A^\circ)''$-linear bounded map, for all $h \in G$.
        \item[(2)] The map $\emph{Der}(\C[G] \subset A \rtimes_\alpha G, \tau)\to (\emph{Der}(A,\tau))_{1 \otimes \alpha_{h}}$ defined $ D \mapsto D_h$ is a surjective, right $(A \otimes A^\circ)''$-linear bounded map, for all $h \in G$.
        \item[(3)] The map $\bigoplus_{g \in G}(\emph{Der}(A,\tau) )_{1\otimes \alpha_{g^{-1}}}\to \emph{Der}(\C[G] \subset A \rtimes_\alpha G, \tau)$ defined $ (d_g)_{g\in G} \to \sum_{g \in G}(d_g)^g$ is an invertible, right $(A \otimes A^\circ)''$-linear bounded map, with the inverse being $D \mapsto (D_g)_{g \in G}$. Furthermore, we obtain 
         $$\dim \emph{Der}(\C[G] \subset A \rtimes_\alpha G,\tau)_{(A \otimes A^\circ)''} = |G|\dim \emph{Der}(A,\tau)_{(A \otimes A^\circ)''}.$$
    \end{enumerate}
\end{thm}
\begin{proof}
    \begin{enumerate}
        \item[]
        \item[(1)] For $d \in \text{Der}(A,\tau)_{1 \otimes \alpha_h}$ with $h \in G$, we have by Lemma~\ref{Lem:Der_A_to_B} that $ d^h \in \text{Der}( \C[G] \subset A \rtimes_\alpha G, \tau)$ and so the map has the correct range. Then, for $m \in (A\otimes A^\circ)''$ and $x \in A \rtimes_\alpha G$, one has 
            \begin{align*}
                [d\cdot_{h^{-1}} m]^h(x)&= J_{\tau \otimes \tau^\circ}(u_e \otimes (u^*_h)^\circ)J_{\tau \otimes \tau^\circ} \left[\sum_{g \in G}(u_g^*\otimes u_g^\circ) [d\cdot_{h^{-1}} m] (\alpha_g(x))\right]\\
                &= J_{\tau \otimes \tau^\circ}(u_e \otimes (u^*_h)^\circ)J_{\tau \otimes \tau^\circ} \left[\sum_{g \in G}(u_g^*\otimes u_g^\circ) J_{\tau \otimes \tau^\circ}[(1\otimes \alpha_{h^{-1}})(m^*)]J_{\tau \otimes \tau^\circ} d (\alpha_g(x))\right],
            \end{align*}
        and by $J_{\tau \otimes \tau^\circ}[(1\otimes \alpha_{h^{-1}})(m^*)]J_{\tau \otimes \tau^\circ} \in ((A \rtimes_\alpha G )\otimes (A \rtimes_\alpha G)^\circ)'$, we can further compute
            \begin{align*}
                [d\cdot_{h^{-1}} m]^h(x)&= J_{\tau \otimes \tau^\circ}m^*J_{\tau \otimes \tau^\circ} \left[J_{\tau \otimes \tau^\circ}(u_e \otimes (u^*_h)^\circ)J_{\tau \otimes \tau^\circ}\sum_{g \in G} (u_g^*\otimes u_g^\circ) d(\alpha_g(x) )\right]\\
            &=J_{\tau \otimes \tau^\circ}m^*J_{\tau \otimes \tau^\circ}d^h(x) \\
            &= (d^h \cdot m)(x).
            \end{align*}
        Hence, the map $d \mapsto d^h$ is right $(A \otimes A^\circ)''$-linear.

        Next, let $Y= X \cup \C[G]$, where $\C\langle X \rangle =A$ and so since $d^h|_{\C[G]} \equiv 0$, we have 
         $$\left \|d^h\right\|_Y^2=\sum_{y \in Y}  \langle d^h(y), d^h(y) \rangle= \sum_{x \in X} \sum_{g,r \in G} \langle  (u_g^*\otimes u_g^\circ )d(\alpha_g (x)), (u_r^*\otimes u_r^\circ) d(\alpha_r(x))\rangle,$$
        and since $d(a) \in L^2(A \otimes A^\circ,\tau \otimes \tau^\circ)(u_e \otimes u_e^\circ)$ for all $a \in A$, the computation above becomes
            \begin{align*}
                \left \|d^h\right\|_Y^2&= \sum_{x \in X} \sum_{g,r \in G}\delta_{g=r} \langle  d(\alpha_g(x)), d(\alpha_r(x))\rangle\\
                &=\sum_{x \in X} \sum_{g \in G} \langle d(\alpha_g(x)), d(\alpha_g(x))\rangle\\
                &= \sum_{g \in G} \| d\|^2_{\alpha_g(X)}\\
                &\leq C \|d\|^2_X.
            \end{align*}
        The inequality follows from $\| \cdot \|_X$ is norm equivalent to $\| \cdot \|_{\alpha_g(X)}$. Thus the map in $d \mapsto d^h$ is bounded.
    
        Lastly, since $d(\alpha_g(x)) \in L^2(A \otimes A^\circ,\tau \otimes \tau^\circ)(u_e \otimes (u_e)^\circ)$ for all $g \in G$ and $x \in A$, one has 
            \begin{align}\label{Eq:Injectivity}
                (d^h)_h(x) &= J_{\tau \otimes \tau^\circ}(u_e \otimes u_h^\circ)J_{\tau \otimes \tau^\circ}p_{e,h}\left[ J_{\tau \otimes \tau^\circ}(u_e \otimes (u^*_h)^\circ)J_{\tau \otimes \tau^\circ} \sum_{g \in G} (u_g^* \cdot d(\alpha_g(x)) \cdot u_g)\right]\nonumber\\
                &=J_{\tau \otimes \tau^\circ}(u_e \otimes u_h^\circ)J_{\tau \otimes \tau^\circ}\left[J_{\tau \otimes \tau^\circ}(u_e \otimes (u^*_h)^\circ)J_{\tau \otimes \tau^\circ}d(x)\right]\\
                &= d(x).\nonumber
            \end{align}
        Thus if $d^h=0$, it follows that $d=0$. Hence, the map $d \mapsto d^h$ is injective.
        \item[(2)] For $D \in \text{Der}(\C[G] \subset A \rtimes_\alpha G,\tau)$, we have by Lemma \ref{Lem:Der_B_to_A} that $D_h \in (\text{Der}(A,\tau))_{1\otimes \alpha_h}$, where $h \in G$. For $m  \in (A \otimes A^\circ)''$ and $x \in A$, one has
            \begin{align*}
                (D \cdot_h m )_h(x) &= J_{\tau \otimes \tau^\circ}( u_e \otimes u_h^\circ )J_{\tau \otimes \tau^\circ}[p_{e,h} (D \cdot_h m)(x)]\\
                &= J_{\tau \otimes \tau^\circ}( u_e \otimes u_h^\circ )J_{\tau \otimes \tau^\circ}p_{e,h} J_{\tau \otimes \tau^\circ} [(1 \otimes \alpha_h )(m^*)] J_{\tau \otimes \tau^\circ} D(x),
            \end{align*}
        and using Lemma \ref{Lem:rel_of_p_gh}(2), following the above computation we have
            \begin{align*}
                (D \cdot_h m )_h(x)&= J_{\tau \otimes \tau^\circ}( u_e \otimes u_h^\circ )J_{\tau \otimes \tau^\circ} J_{\tau \otimes \tau^\circ} [(1 \otimes \alpha_h )(m^*)] J_{\tau \otimes \tau^\circ} p_{e,h} D(x)\\
                &=(J_{\tau \otimes \tau^\circ}m^* J_{\tau \otimes \tau^\circ} )[ J_{\tau \otimes \tau^\circ}( u_e \otimes u_h^\circ )J_{\tau \otimes \tau^\circ}p_{e,h}D(x)]\\
                &=(J_{\tau \otimes \tau^\circ}m^* J_{\tau \otimes \tau^\circ} )[D_h(x)]\\
                &= (D_h \cdot m)(x).
            \end{align*}
        Hence, the map in $D \mapsto D_h$ is right $(A\otimes A^\circ)''$-linear.

        Next, let $Y= X \cup \C[G]$ satisfying $\C\langle X \rangle =A$ and so we have
            \begin{align*}
                \left \|D_h\right\|^2_X&=\sum_{x \in X}  \langle D_h(x), D_h(x) \rangle\\
                &= \sum_{x \in X}  \langle  p_{e,h}D(x), p_{e,h}D(x)\rangle\\
                &\leq \|p_{e,h}\|\sum_{x \in X}  \langle  D(x), D(x)\rangle\\
                &=\|D\|_Y^2.
            \end{align*}
        Note that the last equality comes from $D(u_g)=0$ for all $g \in G$. Hence the map in $D \mapsto D_h$ is a bounded map. 
    
        Lastly, lets show that the map $D \mapsto D_h$ is surjective. By Lemma \ref{Lem:Der_A_to_B}, we have $d^h \in \text{Der}(\C[G] \subset A \rtimes_\alpha G,\tau)$, whenever $d \in (\text{Der}(A,\tau))_{1 \otimes \alpha_h}$. By equation (\ref{Eq:Injectivity}) above, it follows that $(d^h)_h=d$.
        \item[(3)]We want to show that $D = \sum_{h\in G}(D_h)^h$ whenever $D \in \text{Der}(\C[G] \subset A \rtimes_\alpha G,\tau)$. For $a \in A$, one has
         $$\sum_{h \in G} (D_h)^h(a)= \sum_{h \in G} J_{\tau \otimes \tau^\circ} (u_e \otimes (u_h^*)^\circ)J_{\tau \otimes \tau^\circ} \sum_{g \in G} (u_g^*\otimes u_g^\circ)[J_{\tau \otimes \tau^\circ} (u_e \otimes u_h^\circ)J_{\tau \otimes \tau^\circ} p_{e,h} D(\alpha_g(a))],$$
        and by $ J_{\tau \otimes \tau^\circ} (u_e \otimes (u_h^*)^\circ)J_{\tau \otimes \tau^\circ} \in ((A\rtimes_\alpha G) \otimes (A\rtimes_\alpha G)^\circ)'$ and $D\in \text{Der}(\C[G] \subset A \rtimes_\alpha G,\tau),$ the above computation becomes
         $$ \sum_{h \in G} (D_h)^h(a)=\sum_{h \in G}  \sum_{g \in G} (u_g^*\otimes u_g^\circ)  p_{e,h} D(\alpha_g(a))=\sum_{h \in G}  \sum_{g \in G} (u_g^*\otimes u_g^\circ)  p_{e,h}  (u_g \otimes (u_g^*)^\circ) D(a).$$
        By Lemma \ref{Lem:rel_of_p_gh}, we futher compute to get
            \begin{align*}
                \sum_{h \in G} (D_h)^h(a)&=\sum_{h \in G}  \sum_{g \in G} (u_g^*\otimes u_g^\circ) (u_g \otimes (u_g^*)^\circ) p_{g^{-1},hg}  D(a)\\
                &=\sum_{h \in G}  \sum_{g \in G}  p_{g^{-1},hg}  D(a)\\
                &= D(a).    
            \end{align*}
        For $b = \sum_{k \in G} a_ku_k$, where $a_k \in A$, we use that $D$ and $\sum_{g \in G} (D_g)^g$ are both vanishing on $\C[G]$ to get
         $$D(b) = \sum_{k \in G} D(a_k)u_k=\sum_{k \in G} \sum_{g \in G} (D_g)^g(a_k)u_k = \sum_{g \in G} (D_g)^g(b).$$ 
    
        Next, we show that given $(d_g)_{g \in G} \in \bigoplus_{g \in G}(\text{Der}(A,\tau) )_{1\otimes \alpha_g}$, one has $\left(\sum_{g \in G} (d_g)^g\right)_h = d_h$. For $a \in A$, one gets
            \begin{align*}
                \left(\sum_{g \in G} (d_g)^g\right)_h(a)& = J_{\tau \otimes \tau^\circ} (u_e \otimes u_h^\circ)J_{\tau \otimes \tau^\circ}p_{e,h}\left[ \sum_{g \in G} J_{\tau \otimes \tau^\circ} (u_e \otimes (u_g^*)^\circ)J_{\tau \otimes \tau^\circ} \sum_{k \in G} (u^*_k \otimes u_k^\circ)d_g(\alpha_k(a))\right]\\
                &= (J_{\tau \otimes \tau^\circ} (u_e \otimes u_h^\circ)J_{\tau \otimes \tau^\circ})p_{e,h}\left[ \sum_{g \in G} \sum_{k \in G} (u^*_k \otimes (u_gu_k)^\circ)(1 \otimes \alpha_g)d_g(\alpha_k(a))\right],
            \end{align*}
        and using $d_g(a) \in L^2(A \otimes A^\circ)(u_e \otimes u_e^\circ)$ for all $g \in G$, the computation above becomes
         $$ \left(\sum_{g \in G} (d_g)^g\right)_h(a)= (J_{\tau \otimes \tau^\circ} (u_e \otimes u_h^\circ)J_{\tau \otimes \tau^\circ}) (u_e \otimes u_h^\circ) (1 \otimes \alpha_h)d_h(a)= d_h(a).$$
        Thus, we have the following right $(A \otimes A^\circ)''$-module isomorphism
         $$ \text{Der}(\C[G] \subset A \rtimes_\alpha G,\tau) \cong \bigoplus_{g \in G} (\text{Der}(A,\tau))_{1 \otimes \alpha_g},$$
        where the direct sum is with respect to the $\langle \cdot, \cdot \rangle_Y$, where $Y = X \cup \C[G]$ satisfying $ A= \C \langle X\rangle$.
        Taking the dimension, one have 
            \begin{align*}
                \dim( \text{Der}(\C[G] \subset A \rtimes_\alpha G,\tau))_{(A \otimes A^\circ)''} &=\dim \left(   \bigoplus_{g \in G} (\text{Der}(A,\tau))_{1 \otimes \alpha_g}\right)_{(A \otimes A^\circ)''}\\
                &=\sum_{g \in G} \dim((\text{Der}(A,\tau))_{1 \otimes \alpha_g})_{(A\otimes A^\circ)''}\\
                &=|G| \dim\text{Der}(A,\tau)_{(A\otimes A^\circ)''}.
            \end{align*} 
        The last equality comes from $1 \otimes \alpha_g$ being a trace-preserving automorphism on $(A \otimes A^\circ)''$. \qedhere
    \end{enumerate} 
\end{proof}

Using the formula obtained in Theorem \ref{Thm:Bij_Map}, we are able relate the dimension of the derivation space of $A\rtimes_\alpha G$ as a $((A\rtimes_\alpha G) \otimes (A\rtimes_\alpha G)^\circ)''$-module with the dimension of the derivation space of $A$ as a $(A \otimes A^\circ)''$-module. To do so, we decompose the derivation space of $A \rtimes_\alpha G$ using Lemma \ref{Lem:Decomp} with the finite dimensional subalgebra $\C[G]$. The following lemma computes the dimension of the derivation space of $\C[G]$. 

\begin{lem}\label{Lem:Der_C[G]}
For $G$ a finite group, we have that 
 $$ \dim \emph{Der}(\C[G], \tau)_{L(G) \otimes L(G)^\circ} = 1- \frac{1}{|G|},$$
where $L(G) = \C[G]$ is the group von Neumann algebra of $G$.
\end{lem}
\begin{proof}
Since $\C[G]$ is a finite dimensional, we have that any derivation $d$ is bounded and hence inner, so $\text{Der}(\C[G], \tau)=\text{InnDer}(\C[G], \tau)$. From \cite[Lemma 1.4]{CN22}, we have
 $$\dim \text{Der}(\C[G], \tau)_{L(G) \otimes L(G)^\circ}= 1-\dim L^2_{\C[G]}(\C[G] \otimes \C[G]^\circ,\tau \otimes \tau^\circ).$$
 Let the left regular representation be denoted by $\lambda$ and $\{\delta_g:g \in G\}$ be the usual basis for $L^2(\C[G],\tau)$. We claim that  
 $$L^2_{\C[G]}(\C[G] \otimes \C[G]^\circ,\tau \otimes \tau^\circ)= \text{span} \left\{\sum_{k \in G}\delta_{kh} \otimes \delta_{k^{-1}}^\circ : h \in G\right\}.$$
Indeed, for $g, h \in G$
    \begin{align*}
        \lambda(g) \cdot \left(\sum_{k \in G} \delta_{kh} \otimes \delta_{k^{-1}}^\circ\right) &= \sum_{k \in G} \delta_{gkh} \otimes \delta_{k^{-1}}^\circ\\
        &=\sum_{k \in G} \delta_{kh} \otimes \delta_{k^{-1}g}^\circ\\
        &= \left(\sum_{g \in G}\delta_{hk} \otimes \delta_{k^{-1}}^\circ\right) \cdot \lambda(g).
    \end{align*}
For $x = \sum_{g,h \in G} \alpha_{g,h} \delta_g \otimes \delta_h^\circ \in L^2_{\C[G]}(\C[G] \otimes \C[G]^\circ,\tau \otimes \tau^\circ)$, one has
 $$\alpha_{g,h} = \langle x, \delta_g \otimes \delta_h^\circ\rangle= \langle \lambda(k^{-1}) \cdot ( \lambda(k)\cdot x), \delta_g \otimes \delta_h^\circ\rangle= \langle \lambda(k^{-1}) \cdot  x\cdot \lambda(k), \delta_g \otimes \delta_h^\circ\rangle,$$
and using that we have a trace, the computation above becomes
 $$\alpha_{g,h} = \langle  x, \lambda(k) \cdot\delta_g \otimes \delta_h^\circ \cdot \lambda(k^{-1}) \rangle= \langle   x, \delta_{kg} \otimes \delta_{hk^{-1}}^\circ \rangle = \alpha_{kg, hk^{-1}}.$$
Hence, we have
    \begin{align*}
        L^2_{\C[G]}(\C[G] \otimes \C[G]^\circ,\tau \otimes \tau^\circ) &= \text{span} \left\{ \sum_{k \in G} \delta_{kg} \otimes \delta_{hk^{-1}}^\circ: g,h \in G\right\}\\
        &=\text{span} \left\{ \sum_{k \in G} \delta_{kg} \otimes \delta_{k^{-1}}^\circ: g \in G\right\}.
    \end{align*}
Now set $\displaystyle f_h : = \frac{1}{|G|^{1/2}} \sum_{k \in G} \delta_{kh} \otimes \delta_{k^{-1}}^\circ$, where $h \in G$. This is an orthonormal family, since
    \begin{align*}
        \langle f_h, f_{h'}\rangle &= \frac{1}{|G|} \sum_{k \in G} \sum_{g \in G} \langle \delta_{kh'} \otimes \delta_{k^{-1}}^\circ, \delta_{gh} \otimes \delta_{g^{-1}}^\circ\rangle\\
        &=\frac{1}{|G|} \sum_{k \in G} \langle \delta_{kh}, \delta_{kh'}\rangle\\
        &= \delta_{h=h'}.
    \end{align*}
Thus we have $[L^2_{\C[G]}(\C[G] \otimes \C[G]^\circ,\tau \otimes \tau^\circ)]= \sum_{h \in G} \langle \cdot , f_h \rangle f_h \in L(G) \otimes L(G)^\circ$. Hence,
    \begin{align*}
        \dim L^2_{\C[G]}(\C[G] \otimes \C[G]^\circ,\tau \otimes \tau^\circ)&= \tau\otimes \tau^\circ\left(\sum_{h \in G} \langle \cdot , f_h \rangle f_h \right)\\
        &= \sum_{ h \in G} \langle (\langle \delta_e \otimes \delta_e^\circ, f_h\rangle) f_h, \delta_e \otimes \delta_e^\circ \rangle\\
        &= \langle \delta_e \otimes \delta_e^\circ, f_e\rangle \langle  f_e, \delta_e \otimes \delta_e^\circ \rangle\\
        &= \frac{1}{|G|}.\qedhere
    \end{align*}
\end{proof}

\begin{cor}\label{Cor:Bij_Map_Formula}
Let $(M,\tau)$ be a tracial von Neumann algebra with a finitely generated unital $*$-subalgebra $A \subset M$. Let $\alpha$ be a trace-preserving action of a finite group $G$ on $(M, \tau)$ such that $A$ is globally invariant under $\alpha$. Then 
 $$\dim \emph{Der}( \C[G] \subset A \rtimes_\alpha G,\tau)_{((A\rtimes_\alpha G) \otimes (A\rtimes_\alpha G)^\circ)''} = \frac{1}{|G|}\dim \emph{Der}(A,\tau)_{(A \otimes A^\circ)''}.$$
Furthermore, one has
 $$\dim \emph{Der}( A \rtimes_\alpha G,\tau)_{((A\rtimes_\alpha G) \otimes (A\rtimes_\alpha G)^\circ)''} -1 = \frac{1}{|G|}(\dim \emph{Der}(A,\tau)_{(A \otimes A^\circ)''}-1).$$
\end{cor}
\begin{proof} Since $[(A\rtimes_\alpha G) \otimes (A\rtimes_\alpha G)^\circ) '' : (A\otimes A^\circ)'']=|G|^2$ and using the formula in Theorem \ref{Thm:Bij_Map}, 
    \begin{align*}
        \dim\text{Der}(\C[G] \subset A \rtimes_\alpha G,\tau)_{((A \rtimes_\alpha G) \otimes (A \rtimes_\alpha G)^\circ)''} &= \frac{1}{|G|^2}\dim\text{Der}(\C[G] \subset A \rtimes_\alpha G,\tau)_{(A \otimes A ^\circ)''} \\
        &= \frac{1}{|G|^2} (|G| \dim\text{Der}(A,\tau)_{(A \otimes A ^\circ)''} )\\
        &= \frac{1}{|G|}\dim \text{Der}(A,\tau)_{(A \otimes A^\circ)''}.
    \end{align*} 
Thus, using Lemma \ref{Lem:Decomp}, Lemma \ref{Lem:Der_C[G]} and the above, one has
    \begin{align*}
        \dim\text{Der}(A \rtimes_\alpha G, \tau&)_{((A\rtimes_\alpha G) \otimes (A\rtimes_\alpha G)^\circ)''}\\
        &= \dim\text{Der}(\C[G], \tau)_{\C[G]\otimes \C[G]} + \dim\text{Der}(\C[G] \subset A \rtimes_\alpha G,\tau)_{((A \rtimes_\alpha G) \otimes (A \rtimes_\alpha G)^\circ)''}\\
        &= \left(1-\frac{1}{|G|}\right)+ \frac{1}{|G|}\dim \text{Der}(A,\tau)_{(A \otimes A^\circ)''}\\
        &= 1+\frac{1}{|G|} (\dim \text{Der}(A,\tau)_{(A \otimes A^\circ)''}-1).\qedhere
    \end{align*}
\end{proof}

\begin{cor}\label{Cor:Subgroup}
Let $(M,\tau)$ be a tracial von Neumann algebra with a finitely generated unital $*$-subalgebra $A \subset M$. Let $\alpha$ be a trace-preserving action of a finite group $G$ on $(M, \tau)$ such that $A$ is globally invariant under $\alpha$. If $H \subset G$ is a finite subgroup of $G$, then 
 $$\dim \emph{Der}( A \rtimes_\alpha G,\tau)_{((A\rtimes_\alpha G) \otimes (A\rtimes_\alpha G)^\circ)''} -1= \frac{1}{[G: H]}(\dim \emph{Der}( A\rtimes_\alpha H,\tau)_{((A\rtimes_\alpha H) \otimes (A\rtimes_\alpha H)^\circ)''}-1 ).$$
\end{cor}
\begin{proof}
From Corollary \ref{Cor:Bij_Map_Formula}, we have both 
 $$\dim \text{Der}( A \rtimes_\alpha G,\tau)_{((A\rtimes_\alpha G) \otimes (A\rtimes_\alpha G)^\circ)''} -1 = \frac{1}{|G|}(\dim \text{Der}(A,\tau)_{(A \otimes A^\circ)''}-1),$$
and
 $$\dim \text{Der}( A \rtimes_\alpha H,\tau)_{((A\rtimes_\alpha H) \otimes (A\rtimes_\alpha H)^\circ)''} -1 = \frac{1}{|H|}(\dim \text{Der}(A,\tau)_{(A \otimes A^\circ)''}-1).$$ 
Then 
    \begin{align*}
        \dim \text{Der}( A \rtimes_\alpha G,\tau)_{((A\rtimes_\alpha G) \otimes (A\rtimes_\alpha G)^\circ)''} -1 &= \frac{1}{|G|}(\dim \text{Der}(A,\tau)_{(A \otimes A^\circ)''}-1) \\
        &= \frac{|H|}{|G|} \left( \frac{1}{|H|} (\dim \text{Der}(A,\tau)_{(A \otimes A^\circ)''}-1) \right)\\    
        &= \frac{1}{[G:H]} (\dim \text{Der}( A \rtimes_\alpha H,\tau)_{((A\rtimes_\alpha H) \otimes (A\rtimes_\alpha H)^\circ)''} -1).
    \end{align*}
The last equality comes from using that $[G:H]= |G||H|^{-1}$.
\end{proof}

\begin{rem}
From \cite[Remark 1.6]{CN22}, we know that 
    \begin{align*}
        \beta^{(2)}_0(A,\tau) &= 1 - \dim\text{InnDer} (A,\tau)_{(A \otimes A^\circ)''},\\
        \beta^{(2)}_1(A,\tau) &= \dim\text{Der} (A,\tau)_{(A \otimes A^\circ)''} -\dim\text{InnDer} (A,\tau)_{(A \otimes A^\circ)''}.
    \end{align*}
Thus by Corollary \ref{Cor:Bij_Map_Formula}, we have 
 $$ \beta_1^{(2)}(A \rtimes_\alpha G,\tau) - \beta_0^{(2)}(A \rtimes_\alpha G,\tau) =\frac{1}{|G|} (\beta_1^{(2)}(A ,\tau) - \beta_0^{(2)}(A,\tau)), $$
where $(A,\tau)$ a tracial $*$-algebra and $G$ a finite group that acts on $A$.
\end{rem}

\section{Schreier's Formula for Free Stein Dimension}
For the reader's convenience, we state Theorem $A$.  
\begin{thm}[{Theorem \ref{Thm:A} }]\label{Thm:Free_Stein_Dim} 
Let $G\stackrel{\alpha}{\curvearrowright}(M,\tau)$ be a trace-preserving action of a finite group $G$ on a tracial von Neumann algebra and let $A\subset M$ be a finitely generated unital $*$-subalgebra which is globally invariant under $\alpha$. Then 
 $$\sigma( \C[G] \subset A \rtimes_\alpha G,\tau) = \frac{1}{|G|}\sigma(A,\tau).$$
Furthermore, we have
 $$\sigma(A \rtimes_\alpha G,\tau) -1 = \frac{1}{|G|}(\sigma(A,\tau)-1).$$
\end{thm}
\begin{proof}
First, we show that $d^h \in \text{Der}_{1\otimes 1}(\C[G] \subset A \rtimes_\alpha G, \tau)$, whenever $d \in \text{Der}_{1 \otimes 1}(A,\tau)$. For $d \in \text{Der}_{1 \otimes 1}(A,\tau)$, we have  $d^h \in \text{Der}(\C[G] \subset A\rtimes_\alpha G)$ by Lemma \ref{Lem:Der_A_to_B}. Since $d \in \text{Der}_{1 \otimes 1}(A,\tau)$, we have $1 \otimes 1^\circ \in \text{dom}(d^*)$. For $b = \sum_{k \in G} a_ku_k \in A \rtimes_\alpha G$, we have 
    \begin{align*}
        \langle u_e \otimes u_e^\circ,d^h(b) \rangle &= \sum_{k \in G}\langle u_e \otimes u_e^\circ,d^h(a_ku_k) \rangle\\
        &= \sum_{k,g\in G} \langle u_e \otimes u_e^\circ , J_{\tau \otimes \tau^\circ}(u_e \otimes (u_h^*)^\circ )J_{\tau \otimes \tau^\circ}(u_g^* \otimes (u_gu_k)^\circ)d(\alpha_g(a_k))\rangle\\
        &= \sum_{k,g \in G} \langle u_g \otimes (u_{h^{-1}k^{-1}g^{-1}})^\circ, d(\alpha_g(a_k))\rangle\\
        &= \sum_{k,g \in G}\delta_{g =e}\delta_{h^{-1}k^{-1}g^{-1}=e}  \langle u_{g} \otimes (u_{h^{-1}k^{-1}g^{-1}})^\circ, d(\alpha_g(a_k))\rangle\\
        &= \langle u_e \otimes u_e^\circ, d(a_{h^{-1}})\rangle
    \end{align*}
and by the identification $1 \otimes 1^\circ =u_e\otimes u_e^\circ \in L^2(A \otimes A^\circ, \tau \otimes \tau^\circ)(u_e \otimes u_e^\circ)$, the above computation becomes
    \begin{align*}
        \langle u_e \otimes u_e^\circ,d^h(b) \rangle&= \langle d^*(1\otimes 1^\circ), a_{h^{-1}}\rangle\\
        &= \langle d^*(1\otimes 1^\circ)u_{h^{-1}}, a_{h^{-1}} u_{h^{-1}}\rangle\\
        &= \langle d^*(1\otimes 1^\circ)u_{h^{-1}}, b\rangle.
    \end{align*}
Hence, $(u_e \otimes u_e^\circ ) \in \text{dom}((d^h)^*)$ with $(d^h)^*(u_e \otimes u_e^\circ) = d^*(1 \otimes 1^\circ)u_h^*$, that is $d^h \in \text{Der}_{1 \otimes 1}(\C[G] \subset A\rtimes_\alpha G)$. Thus the map $d \mapsto d^h$ in Theorem \ref{Thm:Bij_Map} can be restricted to the following subspaces $\text{Der}_{1 \otimes 1}(\C[G] \subset A \rtimes_\alpha G, \tau)$ and $\text{Der}_{1 \otimes 1}(A,\tau)$.

Next, we show that $D_h \in \text{Der}_{1\otimes 1}(A, \tau)$, whenever $D \in \text{Der}_{1 \otimes 1}(\C[G] \subset A\rtimes_\alpha G,\tau)$. By Lemma \ref{Lem:Der_B_to_A}, we have $D_h \in \text{Der}(A,\tau)$, whenever $D \in \text{Der}_{1 \otimes 1}(\C[G] \subset A\rtimes_\alpha G)$. Since  $D \in \text{Der}_{1 \otimes 1}(\C[G] \subset A\rtimes_\alpha G)$, then $A \rtimes_\alpha G \otimes (A \rtimes_\alpha G)^\circ \subset \text{dom}(D^*)$. For  $a \in A$ and using the identification $1 \otimes 1^\circ =u_e\otimes u_e^\circ \in L^2(A \otimes A^\circ, \tau \otimes \tau^\circ)(u_e \otimes u_e^\circ), $ it follows that
    \begin{align*}
        \langle  1\otimes 1^\circ, D_h(a) \rangle &= \langle u_e \otimes u_e^\circ, D_h(a) \rangle\\
        &=\langle u_e \otimes u_e^\circ, J_{\tau \otimes \tau^\circ}(u_e \otimes u_h^\circ) J_{\tau \otimes \tau^\circ} p_{e,h}D(a) \rangle\\
        &= \langle u_e \otimes u_h^\circ,  p_{e,h}D(a) \rangle\\
        &= \langle D^*(u_e \otimes u_h^\circ),  a \rangle\\
        &=\langle [L^2(A,\tau)]  D^*(u_e \otimes u_h^\circ), a\rangle.
    \end{align*}
Hence we have $1 \otimes 1^\circ \in \text{dom}(D^*)_h$ with $(D_h)^*(1 \otimes 1^\circ) = [L^2(A,\tau)]D^*(u_e\otimes u_h^\circ)$, that is $D_h \in \text{Der}_{1 \otimes 1}(A,\tau)$. Thus the map $D \mapsto D_h$ in Theorem \ref{Thm:Bij_Map} can be restricted to the following subspaces $\text{Der}_{1 \otimes 1}(\C[G] \subset A \rtimes_\alpha G, \tau)$ and $\text{Der}_{1 \otimes 1}(A,\tau)$.

By applying the third map in Theorem \ref{Thm:Bij_Map}, one has the right $(A \otimes A^\circ)''$-module isomorphism
    $$ \text{Der}_{1 \otimes 1}(\C[G] \subset A \rtimes_\alpha G,\tau) \cong \bigoplus_{g \in G} (\text{Der}_{1 \otimes 1}(A,\tau))_{1 \otimes \alpha_g},$$
where the direct sum is with respect to the $\langle \cdot, \cdot \rangle_Y$, where $Y = X \cup \C[G]$ satisfying $ A= \C \langle X\rangle$.

Lastly, from Lemma \ref{Lem:Decomp} and $\overline{\text{Der}_{1 \otimes 1}(\C[G],\tau)}= \overline{\text{InnDer}(\C[G],\tau)}$, we can apply Corollary \ref{Cor:Bij_Map_Formula} to the following subspaces $\overline{\text{Der}_{1 \otimes 1} (A\rtimes_\alpha G,\tau))} \subset \text{Der}(A\rtimes_\alpha G,\tau)$ and $\overline{\text{Der}_{1 \otimes 1}(A,\tau))} \subset \text{Der}(A,\tau)$. Thus we get
    \begin{align*}
        \sigma( \C[G] \subset A \rtimes_\alpha G, \tau) &= \dim\overline{\text{Der}_{1 \otimes 1}(\C[G] \subset A \rtimes_\alpha G,\tau)}_{((A\rtimes_\alpha G) \otimes (A\rtimes_\alpha G)^\circ)''} \\
        &= \frac{1}{|G|}  \dim \overline{\text{Der}_{1 \otimes 1}(A,\tau)}_{(A \otimes A^\circ)''}\\
        &= \frac{1}{|G|} \sigma(A,\tau),
    \end{align*} 
and
    \begin{align*}
        \sigma(A \rtimes_\alpha G, \tau)-1 &= \dim\overline{\text{Der}_{1 \otimes 1}(A \rtimes_\alpha G,\tau)}_{((A\rtimes_\alpha G) \otimes (A\rtimes_\alpha G)^\circ)''} -1\\
        &= \frac{1}{|G|} ( \dim \overline{\text{Der}_{1 \otimes 1}(A,\tau)}_{(A \otimes A^\circ)''} -1)\\
        &= \frac{1}{|G|} (\sigma(A,\tau)-1). \qedhere
    \end{align*} 
\end{proof}

The proof of the following corollary is similar to the proof of Corollary \ref{Cor:Subgroup}.
\begin{cor}
Let $G\stackrel{\alpha}{\curvearrowright}(M,\tau)$ be a trace-preserving action of a finite group $G$ on a tracial von Neumann algebra and let $A\subset M$ be a finitely generated unital $*$-subalgebra which is globally invariant under $\alpha$. If $H \subset G$ is a finite subgroup of $G$, then 
 $$\sigma( A \rtimes_\alpha G,\tau)-1 = \frac{1}{[G: H]}( \sigma ( A\rtimes_\alpha H,\tau)-1).$$
\end{cor}

\begin{rem}
If we take the hypotheses of Theorem \ref{Thm:Free_Stein_Dim} and additionally the group action is trivial, then we have $A \rtimes_\alpha G \cong A \otimes \C[G]$ and we recover Theorem 2.5 in \cite{CN22}.
\end{rem}

Even though Theorem \ref{Thm:Free_Stein_Dim} is limited to finite groups, we can apply it to infinite groups with the assumption that $G$ has an abundance of subgroups. Such examples of groups exist since by \cite[Theorem IV]{HNN49} the authors show that any countable group can be embedded into a group that is generated by two elements. Consequently, one can consider a group $G$ with two generators such that $\bigoplus_{n \in \N} \Z/n\Z$ is embedded into $G$, and such a group satisfies the condition of the corollary below.

\begin{cor}\label{Cor:Many_Subgroups}
Let $(M, \tau)$ be a tracial von Neumann algebra with a finitely generated unital $*$-algebra $A \subset M$. Let $G$ be a finitely generated group such that for all $n \in \N$, there exists a subgroup $G_n$ with $ n \leq |G_n| < \infty.$ Let $\alpha$ be a trace-preserving action of group $G$ on $(M,\tau)$ and $A$ be globally invariant under $\alpha$. Then
$$ \sigma(A \rtimes_\alpha G,\tau)\leq \beta_1^{(2)}(G) -\beta_0^{(2)}(G) +1 .$$
\end{cor}
\begin{proof}
Observe that $A \rtimes_\alpha G = ( A \rtimes_\alpha G_n) \vee \C[G]$ and by \cite[Corollary 2.13]{CN21}, we have
$$ \sigma(\C[G_n] \subset A \rtimes_\alpha G,\tau) \leq \sigma(\C[G_n] \subset A \rtimes_\alpha G_n,\tau) + \sigma( \C[G_n] \subset \C[G],\tau).$$
Then using Lemma \ref{Lem:Decomp} and the above inequality, the computation becomes
\begin{align*}
    \sigma(A \rtimes_\alpha G,\tau) &= \sigma(\C[G_n] \subset A \rtimes_\alpha G,\tau) + \sigma(\C[G_n],\tau)\\
    &\leq \sigma(\C[G_n] \subset A \rtimes_\alpha G_n,\tau) + \sigma( \C[G_n] \subset \C[G],\tau) + \sigma(\C[G_n],\tau).
\end{align*}
Now using Lemma \ref{Lem:Decomp} again, we can further establish
$$ \sigma(A \rtimes_\alpha G,\tau)= \sigma( A \rtimes_\alpha G_n,\tau) + \sigma(\C[G],\tau) - \sigma(\C[G_n],\tau). $$
By Theorem \ref{Thm:Free_Stein_Dim}, Proposition 5.1 in \cite{CN21} and Lemma \ref{Lem:Der_C[G]}, we have
\begin{align*}
\sigma(A \rtimes_\alpha G,\tau)  &\leq (1+\frac{1}{|G_n|}(\sigma( A,\tau) -1)) + \beta_1^{(2)}(G) - \beta_0^{(2)}(G) +1 -\left(1- \frac{1}{|G_n|}\right)\\
 &= \frac{1}{|G_n|} \sigma(A,\tau) +\beta_1^{(2)}(G) - \beta_0^{(2)}(G) +1.
\end{align*}
 Thus taking the limit as $n$ goes to infinity, we have
\begin{equation*}
    \sigma(A \rtimes_\alpha G,\tau)\leq \beta_1^{(2)}(G) - \beta_0^{(2)}(G) +1. \qedhere
\end{equation*}
\end{proof}

\begin{ex}\label{Ex:Fin_Gen_gp_Sigma}
Let $G$ be a countable abelian group such that for all $n \in \N$, there exists a subgroup $G_n$ with $ n \leq |G_n| < \infty.$ Then $L(G)$ is a separable abelian von Neumann algebra and there exists a finite self-adjoint set $Y \subset L(G)$ such that $L(G) = \C \langle Y\rangle''$. Let $G \stackrel{\alpha}{\curvearrowright}(M,\tau)$ be a trace-preserving action on a tracial von Neumann algebra. Suppose $A \subset M$ is a finitely generated unital $*$-subalgebra such that $  M =A''$ and $A$ is globally invariant under $\alpha$. For $G_n \subset G$ a finite subgroup, we have that $A\vee \C\langle Y\rangle  \subset (A \rtimes_\alpha G_n) \vee \C\langle Y \rangle$ and so $((A \rtimes_\alpha G_n) \vee \C\langle Y \rangle)'' = M \rtimes_\alpha G$. By Lemma \ref{Lem:Decomp} and \cite[Corollary 2.13]{CN21}, one gets 
\begin{align*}
    \sigma( (A\rtimes_\alpha G_n) \vee \C\langle Y\rangle, \tau ) &= \sigma( \C[G_n] \subset (A\rtimes_\alpha G_n) \vee \C\langle Y\rangle,\tau) + \sigma( \C[G_n], \tau)\\
    &\leq \sigma(\C[G_n] \subset A\rtimes_\alpha G_n,\tau) + \sigma(\C[G_n] \subset (\C[G_n] \vee \C\langle Y\rangle),\tau) + \sigma( \C[G_n], \tau).
\end{align*}
Using Lemma \ref{Lem:Decomp} again, one has
$$\sigma( (A\rtimes_\alpha G_n) \vee \C\langle Y\rangle, \tau ) = \sigma(A\rtimes_\alpha G_n,\tau) + \sigma( (\C[G_n] \vee\C\langle Y\rangle),\tau) -\sigma( \C[G_n], \tau),$$
and so by Theorem \ref{Thm:Free_Stein_Dim}, \cite[Corollary 3.3]{CN22} and Lemma \ref{Lem:Der_C[G]}, one has 
\begin{align*}
    \sigma( (A\rtimes_\alpha G_n) \vee \C\langle Y\rangle, \tau ) & \leq  (1 + \frac{1}{|G_n|}( \sigma(A,\tau) -1)) + 1 - \left(1 - \frac{1}{|G_n|}\right)\\
    &= 1 +\frac{1}{|G_n|} \sigma(A,\tau).
\end{align*} 
Hence, for all $\varepsilon>0$, $M \rtimes_\alpha G$ admits a dense $*$-subalgebra $B$ with $\sigma(B,\tau) \leq 1 + \varepsilon$. $\hfill\blacksquare$
\end{ex}

\section{Schreier's Formula for $\dim\text{Der}_c(A.\tau)$}
We fix $A \subset M$ a finitely generated unital $*$-subalgebra with $A=\C\langle X\rangle$, where $X$ is a finite self-adjoint subset of $A$. In \cite{Shl09}, Shlyakhtenko considered the following subspace of the derivation space
$$\text{Der}_c(A ,\tau) := \{ d \in \text{Der}(A,\tau) : d^*(1 \otimes 1) \in A , d(x) \in A\otimes A^\circ \text{ for all }x\in X\}.$$
Notice that by definition, we have $\text{Der}_c(A,\tau) \subset \text{Der}_{1 \otimes 1}(A,\tau)$ and 
$$ \dim \overline{\text{Der}_c(A,\tau)} \leq \sigma(A,\tau).$$
We mention that if $\xi \in A \otimes A^\circ$, then the derivation $d (\cdot) :=[ \cdot, \xi]$ is in $\text{Der}_c(A,\tau),$
and so 
$$ \text{InnDer}(A,\tau) \subset \overline{\text{Der}_c(A,\tau)}.$$

The following result shows that we can apply Theorem \ref{Thm:Bij_Map} and Corollary \ref{Cor:Bij_Map_Formula} to this subspace, but first we remind the reader on notation of the free difference quotients and their connection to derivations on $A$. Let $T_X = \{ t_x: x \in X\}$ be a set of indeterminate variables equipped with the involution $t^*_x = t_{x^*}$. We denote by $\C\langle T_X \rangle$ the $*$-algebra formally spanned by elements of the form $t_{x_1} t_{x_2} \cdots t_{x_d}$, where $x_i \in X$ for $i =\{1, \ldots , d\}$. Let $\text{ev}_X:\C \langle T_X \rangle \to A$ be the $*$-homomorphism extended linearly by $t_{x_1} t_{x_2} \cdots t_{x_d} \mapsto x_1x_2 \cdots x_d$. Given a $p \in \C \langle T_X \rangle$, we write $p(X)$ for $\text{ev}_X(p)$. This map is surjective, and it is injective if and only if $X$ is algebraically free. The \textit{free difference quotients} are the derivations $\partial_x: \C \langle T_X \rangle \to \C \langle T_X \rangle \otimes \C \langle T_X \rangle^\circ$ defined by linearity and the conditions
\begin{align*}
    \partial_x(t_y) &=\delta_{x=y} 1 \otimes 1,\\
    \partial_x(pq) &=p\cdot\partial_x(q)+ \partial_x(p)\cdot q.
\end{align*}
Given $p,q\in \C \langle T_X\rangle$, we denote $(p \otimes q)(X) =p(X) \otimes q(X)$ and by the above, we have $(\partial_x(p))(X) \in A \otimes A^\circ$ for $p \in \mathbb{C}\langle T_X\rangle$. Given $d \in \text{Der}(A,\tau)$ and $p \in \C\langle T_x \rangle$, one has
\begin{equation}\label{Eq:Der_on_A}
    d(p)= \sum_{x \in X} \partial_x(p)(X) d(x).
\end{equation} 
Notice that to define a derivation on $A$, one can, using the right hand side of equation (\ref{Eq:Der_on_A}), define a derivation on $\C \langle T_X\rangle$ to $L^2(A \otimes A^\circ, \tau \otimes \tau^\circ)$ and check that the derivation factors through $A$.

\begin{thm}[{Theorem \ref{Thm:B} }]\label{Thm:Anot_subs}
Let $G\stackrel{\alpha}{\curvearrowright}(M,\tau)$ be a trace-preserving action of a finite group $G$ on a tracial von Neumann algebra and let $A\subset M$ be a finitely generated unital $*$-subalgebra which is globally invariant under $\alpha$. Then
 $$\dim \overline{\emph{Der}_c( \C[G] \subset A \rtimes_\alpha G,\tau )}_{((A\rtimes_\alpha G) \otimes (A\rtimes_\alpha G)^\circ)''} = \frac{1}{|G|}\dim\overline{\emph{Der}_c(A,\tau)}_{(A \otimes A^\circ)''}.$$
Furthermore, we have
 $$\dim \overline{\emph{Der}_c(A \rtimes_\alpha G, \tau)}_{((A\rtimes_\alpha G) \otimes (A\rtimes_\alpha G)^\circ)''} -1 = \frac{1}{|G|}(\dim \overline{\emph{Der}_c(A,\tau)}_{(A \otimes A^\circ)''}-1).$$
\end{thm}
\begin{proof}
First, we show that $d^h \in \text{Der}_c( \C[G] \subset A \rtimes_\alpha G, \tau)$, whenever $d \in \text{Der}_c(A,\tau)$ and $h \in G$. For $d \in \text{Der}_c(A,\tau)$, we have, by Theorem \ref{Thm:Free_Stein_Dim}, that $d^h \in \text{Der}_{1 \otimes 1}(\C[G] \subset A\rtimes_\alpha G)$ with $(d^h)(u_e \otimes u_e^\circ) = d^*(1 \otimes 1^\circ)u_h^* \in (A\rtimes_\alpha G)\otimes(A\rtimes_\alpha G)^\circ $. Let $p_g \in \C \langle T\rangle$ such that $\alpha_g(x) =p_g(X)$, where $g \in G$. Then for $x \in X$, we have 
\begin{align*}
    d^h(x) &= J_{\tau \otimes \tau^\circ}(u_e \otimes (u^*_h)^\circ)J_{\tau \otimes \tau^\circ} \sum_{g \in G} (u_g^* \otimes u_g^\circ)d\big(\alpha_g(a_k)\big)\\
    &=J_{\tau \otimes \tau^\circ}(u_e \otimes (u^*_h)^\circ)J_{\tau \otimes \tau^\circ} \sum_{g \in G}(u_g^* \otimes u_g^\circ) \sum_{x\in X}  \big(\partial_x(p_g)(X)\big)d(x)
\end{align*}
Since $d(x) \in A \otimes A^\circ$ for all $x \in X$, we have that $d^h(x) \in (A\rtimes_\alpha G)\otimes(A\rtimes_\alpha G)^\circ $ and $d^h(u_g)=0$ for all $g \in G$. Hence $d^h \in \text{Der}_c(A \rtimes_\alpha G,\tau ).$

Next, we show that $D_h \in \text{Der}_c(A,\tau)$, whenever $D \in \text{Der}_c(\C[G] \subset A\rtimes_\alpha G,\tau)$. Let $D \in \text{Der}_c(\C[G] \subset A\rtimes_\alpha G, \tau)$ and by Theorem \ref{Thm:Free_Stein_Dim}, we have $D_h \in \text{Der}_{1\otimes1}(A,\tau)$ with $D^*_h(1 \otimes 1^\circ)=[L^2(A,\tau)] D^*(u_e\otimes u_h^*)$, where $h \in G$. Using \cite[Proposition 4.1]{Voi98}, we have
$$D_h^*(1\otimes 1^\circ) = [L^2(A,\tau)] \big(D^*(u_e\otimes u_e^\circ) u_h - (1 \otimes \tau^\circ)D(u_h^*)^*\big).$$
Since $(A\rtimes_\alpha G)\otimes (A \rtimes_\alpha G)^\circ$ decomposes to a finite direct sum of $A \otimes A^\circ$, $D_h^*(1 \otimes 1^{\circ}) \in A \otimes A^\circ$. Since $p_{g,h} \in (A \otimes A)'$ for all $g,h \in G$, we have $D_h(x) \in A \otimes A^\circ$ for all $x \in X$. Hence  $D_h \in \text{Der}_c(A,\tau)$.

By applying the third map in Theorem \ref{Thm:Bij_Map}, one has the right $(A \otimes A^\circ)''$-module isomorphism
    $$ \text{Der}_c(\C[G] \subset A\rtimes_\alpha G, \tau )= \bigoplus_{g \in G} (\text{Der}_c(A,\tau))_{1 \otimes \alpha_g},$$
where the direct sum is with respect to the $\langle \cdot, \cdot \rangle_Y$, where $Y = X \cup \{u_g : g \in G\}$.

Lastly, by Lemma \ref{Lem:Decomp} and $\overline{\text{Der}_c(\C[G],\tau)}= \overline{\text{InnDer}(\C[G],\tau)}$, we can apply Corollary \ref{Cor:Bij_Map_Formula} to the following subspaces $\text{Der}_c(\C[G] \subset A\rtimes_\alpha G,\tau) \subset \text{Der}(A\rtimes_\alpha G,\tau)$ and $\text{Der}_c(A,\tau) \subset \text{Der}(A,\tau)$ to get
$$\dim\overline{\text{Der}_c(\C[G] \subset A\rtimes_\alpha G,\tau)}_{((A\rtimes_\alpha G) \otimes (A\rtimes_\alpha G)^\circ)''}  = \frac{1}{|G|}  \dim \overline{\text{Der}_c(A,\tau)}_{(A \otimes A^\circ)''}$$
and
$$\dim\overline{\text{Der}_c( A\rtimes_\alpha G,\tau)}_{((A\rtimes_\alpha G) \otimes (A\rtimes_\alpha G)^\circ)''} -1= \frac{1}{|G|} ( \dim \overline{\text{Der}_c(A,\tau)}_{(A \otimes A^\circ)''}-1).$$
\end{proof}

The proof of the following corollary is similar to the proof of Corollary \ref{Cor:Subgroup}.
\begin{cor}
Let $G\stackrel{\alpha}{\curvearrowright}(M,\tau)$ be a trace-preserving action of a finite group $G$ on a tracial von Neumann algebra and let $A\subset M$ be a finitely generated unital $*$-subalgebra which is globally invariant under $\alpha$. If $H \subset G$ is a finite subgroup of $G$, then 
 $$\dim \overline{\emph{Der}_c( A \rtimes_\alpha G,\tau)}_{((A\rtimes_\alpha G) \otimes (A\rtimes_\alpha G)^\circ)''}-1 = \frac{1}{[G: H]}( \dim\overline{\emph{Der}_c( A\rtimes_\alpha H,\tau)}_{(A \otimes A^\circ)''}-1).$$
\end{cor}

The following is an estimate for the microstates free entropy dimension $\delta_0$ when we consider the crossed product of a von Neumann algebra with a finite group. This uses a known inequality, $ \dim \text{Der}_c(A,\tau)\leq \delta_0 $ (see \cite[Theorem 2, Corollary 17]{Shl09}). Note that the assumption $A '' \hookrightarrow R^\omega$  is to guarantee $\delta_0 >-\infty$, where $R$ is the hyperfinite II$_1$ factor.
\begin{cor}
Let $G\stackrel{\alpha}{\curvearrowright}(M,\tau)$ be a trace-preserving action of a finite abelian group $G$ on a tracial von Neumann algebra and let $A\subset M$ be a finitely generated unital $*$-subalgebra which is globally invariant under $\alpha$. Assume that $A''$ can be embedded in the ultrapower of the hyperfinite $\emph{II}_1$ factor. Then for any generating set $Y$ of $A \rtimes_\alpha G$, we have
    $$\frac{1}{|G|} (\dim \overline{\emph{Der}_c(A, \tau)}_{(A \otimes A^\circ)''}-1) +1  \leq \delta_0(Y) $$
\end{cor}
\begin{proof}
Using \cite[Theorem 2]{Shl09} and Theorem \ref{Thm:Anot_subs}, we have
\begin{equation*}
\delta_0(Y) \geq \dim \overline{\text{Der}_c(A \rtimes_\alpha G, \tau)}_{((A\rtimes_\alpha G) \otimes (A\rtimes_\alpha G)^\circ)''}= \frac{1}{|G|}(\dim \overline{\text{Der}_c(A, \tau)}_{(A \otimes A^\circ)''}-1)+1.\qedhere
\end{equation*}
\end{proof}

\section{Schreier's Formula for $\Delta$}
In this section, we remind the reader of $\Delta$ and we show that $\Delta$ can be computed by taking the von Neumann dimension of a certain subspace of derivations. Then we show Theorem \ref{Thm:Delta_Dim} by choosing a set of generators to apply Theorem \ref{Thm:Bij_Map} and Corollary \ref{Cor:Bij_Map_Formula} on $\Delta$. We further assume that the finite group $G$ acting on $M$ is \emph{abelian}, since this extra assumption guarantees a generating set of $A$ that is scaled under the action of $\alpha$. Lastly, we apply our results to the free entropy dimensions.

We fix $A \subset M$ a finitely generated unital $*$-subalgebra with $A=\C\langle X\rangle$, where $X$ is a finite self-adjoint subset of $A$. In \cite[Section 3]{CS05}, Connes and Shlyakhtenko defined the quantity $\Delta$ for $X$ with respect to $\tau$ as
 $$ \Delta(A,\tau) = \dim\left(\overline{\partial^t_X\Big(B\big(L^2(A,\tau)\big)\Big)}^{\text{WOT}} \cap \Big(\text{HS}\big(L^2(A,\tau)\big)\Bigg)^{X}\right)_{(A \otimes A^\circ)''},$$
where $\text{HS}\big(L^2(A,\tau)\big)$ is the set of Hilbert--Schmidt operators on $L^2(A,\tau)$ and $\partial^t_X(y) = \big( [y, x]\big)_{x \in X}$ for $y \in B\big(L^2(A,\tau)\big)$. Although it is not obvious from the definition, $\Delta(A,\tau)$ only depends on $A$ (see \cite[Theorem 3.3]{CS05}. Since $\partial^t_{X}$ is continuous with respect to the weak operator topology, we can redefine the quantity as
 $$ \Delta(A,\tau) = \dim\left(\overline{\partial^t_X\Big(\text{HS}\big(L^2(A,\tau)\big)\Big)}^{\text{WOT}} \cap \Big(\text{HS}\big(L^2(A,\tau)\big)\Big)^X\right)_{(A \otimes A^\circ)''}.$$
Notice that $\overline{\partial^t_{X}\Big(\text{HS}\big(L^2(A,\tau)\big)\Big)}^{\text{WOT}} \cap \Big(\text{HS}\big(L^2(A,\tau)\big)\Big)^{X}$ is a closed subspace of $\Big(\text{HS}\big(L^2(A,\tau)\big)\Big)^{X}$, since convergence with respect to the Hilbert--Schmidt norm implies the weak operator topology convergence.  

Recall the identification $L^2(A \otimes A^\circ,\tau \otimes\tau^\circ)$ with $\text{HS}\big(L^2(A,\tau)\big)$ via $ \xi \otimes \eta^\circ \mapsto \xi \otimes \bar\eta$, where $\xi,\eta \in L^2(A,\tau)$ and note that $L^2(A,\tau) \odot L^2(A^\circ, \tau^\circ)$ is identified with $\text{FR}\big(L^2(A,\tau)\big)$. Thus, $S = (S_x)_{x \in X} \in  \overline{\partial^t_{X}\Big(\text{HS}\big(L^2(A,\tau)\big)\Big)}^{\text{WOT}} \cap \Big(\text{HS}\big(L^2(A,\tau)\big)\Big)^{X}$, if there exists a net $\{\xi_\lambda\} \in L^2(A\otimes A^\circ,\tau\otimes\tau^\circ)$ such that 
$$ \lim_{\lambda \to \infty} \langle S - \partial^t_X(\xi_\lambda), \xi \otimes \eta^\circ \rangle= \lim_{\lambda \to \infty} \sum_{x \in X} \langle S_x- [\xi_\lambda,x] , \xi_x \otimes \eta_x^\circ\rangle=0,$$
whenever $\xi, \eta \in L^2(A,\tau)^X$.

Also, observe that $ [\phi_X(\text{Der}(A,\tau))] \in M_{|X|} ((A\otimes A^\circ)'',\tau \otimes \tau^\circ)$, where $\phi_X$ is the map defined in subsection 1.3, because $\phi_X(\text{Der}(A,\tau)) \subset L^2(A \otimes A^\circ, \tau \otimes \tau^\circ)^X$ is invariant under the diagonal action of $(A \otimes A^\circ)'$.

\begin{defi}
For $A = \C\langle X \rangle$, we denote by $\text{Der}_{\text{FR},X}(A,\tau)$ the derivations $d \in \text{Der}(A,\tau)$ such that $ (d(x))_{x \in X} \in [\phi_X(\text{Der}(A,\tau))]( L^2(A,\tau)\odot L^2(A^\circ, \tau^\circ))^{X}$. For each $d \in \text{Der}_{\text{FR},X}(A,\tau)$, we define a seminorm $\rho_{X,d} : \text{Der}(A,\tau) \to \mathbb{R}$ by 
$$\rho_{X,d}( d') := \left| \langle d', d \rangle_X\right|.$$
Set $\mathcal{P}_{\text{FR},X} : = \{ \rho_{X,d}: d \in \text{Der}_{\text{FR},X}(A,\tau)\}$. We define $\text{Der}_{[\, \cdot \,, X]}(A,\tau)$ to be the closure of $\text{InnDer}(A,\tau)$ under the topology generated by this family of seminorms.
\end{defi}

Since $L^2(A,\tau)\odot L^2(A^\circ,\tau^\circ)$ is dense in $L^2(A\otimes A^\circ,\tau\otimes\tau^\circ)$, the intersection of the kernels of these seminorms is trivial. Thus $(\text{Der}(A,\tau), \mathcal{P}_{\text{FR}, X})$ is a locally convex space. By the Cauchy--Schwartz inequality, it follows that $\text{Der}_{[\,\cdot\, , X]}(A,\tau)$ is a closed $(A \otimes A^\circ)''$-submodule in $\text{Der}(A,\tau)$. Although $\text{Der}_{\text{FR}, X}(A,\tau)$ is not closed, it is a $(A \otimes A^\circ)$-submodule. We note that each element in $\text{Der}_{[\, \cdot \,, X]}(A,\tau)$ is ``almost weakly approximated" by inner derivations, in the sense that the weak limits are only against elements in $\text{Der}_{\text{FR},X}(A,\tau)$.  

The following lemma shows that the corresponding subspace of derivations for $\Delta$ is the one we defined above.
\begin{lem}\label{Lem:CS_Correp_Der}
Let $(M,\tau)$ be a tracial von Neumann algebra with finitely generated unital $*$-subalgebra $A \subset M$ and let $X\subset A$ be any finite self-adjoint subset satisfying $A = \C\langle X\rangle$. The following linear map
    \begin{align*}
        \emph{Der}_{[ \,\cdot\, , X]}(A,\tau) &\to \overline{\partial^t_{X}\Big(\emph{HS}\big(L^2(A,\tau)\big)\Big)}^{\emph{WOT}} \cap \Big(\emph{HS}\big(L^2(A,\tau)\big)\Big)^{X}  \\
        d &\mapsto (d(x))_{x \in X} 
    \end{align*}
is bijective, and right $(A \otimes A^\circ)''$-linear. Consequently, $\emph{Der}_{[\, \cdot \, , X]}(A,\tau)$ is a closed right $(A\otimes A^\circ)''$-submodule and we have
$$ \Delta(A,\tau) = \dim \emph{Der}_{[ \,\cdot\, , X]}(A,\tau).$$
\end{lem}
\begin{proof}
First, we show that the map is valued in the correct range. That is, given $d \in \text{Der}_{[ \,\cdot\, , X]}(A,\tau)$, one has $(d(x))_{x\in X}\in \overline{\partial^t_{X}(\text{HS}(L^2(A,\tau)))}^{\text{WOT}} \cap (\text{HS}(L^2(A,\tau)))^{X} $. Since $d \in \text{Der}_{[ \,\cdot\, , X]}(A,\tau)$ there exists a net $\{d_\lambda\}_{\lambda \in \Lambda} \in \text{InnDer}(A,\tau)$ that almost weakly approximates $d$ with $\xi_\lambda\in L^2(A\otimes A^\circ, \tau \otimes \tau^\circ)$ such that $d_\lambda(\cdot ) : = [ \cdot \, , - \xi_\lambda]$ for each $\lambda$.  For $\xi= (\xi_x)_{x \in X}, \eta=(\eta_x)_{x \in X} \in L^2(A, \tau )^X$, we have
\begin{align*}
    \langle (d(x))_{x \in X} - \partial_X^t(x_\lambda), \xi \otimes \eta\rangle &= \sum_{x \in X} \langle d(x) - [\xi_\lambda, x],  \xi_x \otimes \eta_x\rangle\\
    &=\sum_{x \in X} \langle d(x) - d_\lambda(x), \xi_x \otimes \eta_x \rangle\\
    &= \langle (d(x) - d_\lambda(x))_{x \in X}, \xi \otimes \eta \rangle
\end{align*}
By setting $d' \in \text{Der}_{\text{FR},X}(A,\tau)$ such that  $(d'(x))_{x \in X}: = [\phi_X(\text{Der}(A,\tau))] (\xi \otimes \eta)$, the above computation becomes
$$ \langle (d(x))_{x \in X} - \partial_X^t(x_\lambda), \xi \otimes \eta\rangle= \langle (d(x) - d_\lambda(x))_{x \in X}, [\phi_X(\text{Der}(A,\tau))](\xi \otimes \eta) \rangle= \langle d - d_\lambda, d' \rangle_X.$$
Thus we have that $(d(x))_{x \in X} \in \overline{\partial^t_{X}(\text{HS}(L^2(A,\tau)))}^{\text{WOT}} \cap (\text{HS}(L^2(A,\tau)))^{X}.$

Next, from the Leibniz rule, one has that each derivation is determined by its values on $X$. Hence, the map above is injective. The map is right $(A \otimes A^\circ)''$, since the right action on a derivation is pointwise.

Before we prove sujectivity, we show that any element of $\overline{\partial^t_{X}(\text{HS}(L^2(A,\tau)))}^{\text{WOT}} \cap (\text{HS}(L^2(A,\tau)))^{X}$ gives rise to a derivation in $\text{Der}(A,\tau)$. That is, given $S=(S_x)_{x \in X} \in \overline{\partial^t_{X}(\text{HS}(L^2(A,\tau)))}^{\text{WOT}} \cap (\text{HS}(L^2(A,\tau)))^{X}$, there exists $d_S \in \text{Der}(A,\tau)$ such that $(d_S(x))_{x \in X}=S$. For $p \in \C\langle T_X\rangle$, define
$$ \widehat{d}_S (p) = \sum_{x \in X} \partial_x (p)(X) S_x.$$
We claim that $\widehat{d}_S( \ker(\text{ev}_X))=0$. Let $p \in \C \langle T_X \rangle$ be such that $p(X)=0$. Since $(S_x)_{x \in X} \in \overline{\partial^t_{X}(\text{HS}(L^2(A,\tau)))}^{\text{WOT}} \cap (\text{HS}(L^2(A,\tau)))^{X}$, we have that $(S_x)_{x \in X} = (\text{WOT-}\lim_{\lambda\to \infty}[ \xi_\lambda, x])_{x \in X}$ for some net $(\xi_\lambda)_{\lambda \in \Lambda} \subset L^2(A\otimes A^\circ,\tau\otimes\tau^\circ))$. This means that 
 $$\widehat{d}_S(q(X))= \text{WOT-}\lim_{\lambda \to \infty}[\xi_\lambda, q(X)]$$ 
for all $q \in \C\langle T_X\rangle$. Hence, we have $\widehat{d}_S(p(X))=0$, since  $[\xi_\lambda, p(X)] =0$ for all $\lambda$. Since $\widehat{d}_S$ factors through $A$, it follows that there exist $d_S \in \text{Der}(A,\tau)$ such that $\widehat{d}_S (p) = d_S(p(X))$ for all $p \in \C\langle T_X \rangle$. 

Finally, we show that the map is surjective. Let $d_S$ be the derivation on $A$ as defined above, where $S=(S_x)_{x \in X}$. We claim that $d_S \in \text{Der}_{[ \,\cdot\,, X]}(A,\tau)$. Again, since $S \in \overline{\partial^t_{X}(\text{HS}(L^2(A,\tau)))}^{\text{WOT}} \cap (\text{HS}(L^2(A,\tau)))^{X}$, there exists $(\xi_\lambda)_{\lambda \in \Lambda} \subset L^2(A\otimes A^\circ,\tau\otimes\tau^\circ)$ such that $(\partial^t_X(\xi_\lambda))_{x \in X} \to (S_x)_{x \in X}$ in $L^2$-norm. Set $d_\lambda (\cdot )= [ \, \cdot \, ,-\xi_\lambda] \in \text{InnDer}(A,\tau)$. Then for $d' \in \text{Der}_{\text{FR},X}(A,\tau)$,
    \begin{align*}
        \langle d_S - d_\lambda, d' \rangle_X &= \sum_{x \in X}  \langle (d_S(x) - [x, -\xi_\lambda], d'(x) \rangle\\
        &=\sum_{x \in X} \langle (S_x-[\xi_\lambda,x]), d'(x) \rangle\\
        &=\langle S - \partial^t_{X}(\xi_\lambda), (d'(x))_{x \in X} \rangle.
    \end{align*} 
Since $(d'(x))_{x \in X} \in [\phi_X(\text{Der}(A,\tau))](L^2(A ,\tau) \odot L^2(A^\circ, \tau^\circ))^X$, we have $d_S \in \text{Der}_{[ \,\cdot\,, X]}(A, \tau).$ Thus, we have the following right $(A \otimes A^\circ)''$-module isomophism 
$$\text{Der}_{[ \,\cdot\, , X]}(A,\tau) \cong\overline{\partial^t_{X}\Big(\text{HS}\big(L^2(A,\tau)\big)\Big)}^{\text{WOT}} \cap \Big(\text{HS}\big(L^2(A,\tau)\big)\Big)^{X}.$$
It follows that $\text{Der}_{[ \,\cdot\, , X]}(A,\tau)$ is closed and taking their dimension gives 
\begin{equation*}
    \Delta (A ,\tau) =\dim \text{Der}_{[ \,\cdot\, , X]}(A,\tau)_{(A \otimes A^\circ)''}.\qedhere
\end{equation*} 
\end{proof}

Next, fix $\alpha$ a trace-preserving action of a finite abelian group $G$ on $M$ and let $A$ be globally invariant under $\alpha$. With the assumption that $G$ is a finite abelian group, we show that $A$ contains a generating set that is scaled under $\alpha$. Let $\hat{G}$ be the dual group of $G$ and consider the following finite self-adjoint subset of $A$,
\begin{equation*}
     X_{\hat{G}}:= \left\{ \frac{1}{|G|}\sum_{g \in G} \overline{\chi(g)} \alpha_g(x): x \in X, \chi \in \hat{G}\right\}.
\end{equation*}
Since $X_{\hat{G}} \subset A$  and $x = \sum_{\chi \in \hat{G}} \sum_{g \in G} |G|^{-1}\overline{\chi(g)} \alpha_g(x)$, we have $ \C\langle X_{\hat{G}}\rangle = A$. We note that $X_{\hat{G}}$ is scaled under $\alpha$, since for $\chi \in \hat{G}$, $h \in G$ and $x \in X$
$$\alpha_h \left(\frac{1}{|G|}\sum_{g \in G} \overline{\chi(g)} \alpha_g(x)\right) = \frac{1}{|G|}\sum_{g \in G} \overline{\chi(g)} \alpha_{hg}(x)= \overline{\chi(h^{-1})} \frac{1}{|G|}\sum_{g \in G}\overline{\chi(g)} \alpha_g(x)=\chi(h) \frac{1}{|G|}\sum_{g \in G}\overline{\chi(g)} \alpha_g(x).$$ It follows that  $A \rtimes_\alpha G$ is generated by $Y  = X_{\hat{G}} \cup \{u_g: g\in G\}$, and these generators are also scaled under the action of $\alpha$. Notice that since $Y$ is scaled by $\alpha$, one has $\langle \cdot, \cdot \rangle_Y = \langle \cdot , \cdot \rangle_{\alpha_g(Y)}$ for all $g \in G$. Thus, we can always choose $X$ a finite self-adjoint generating set for $A$ such that $X$ is scaled under $\alpha$ and similarly for $A\rtimes_\alpha G$.

Our objective is to prove that Theorem \ref{Thm:Bij_Map} and Corollary \ref{Cor:Bij_Map_Formula} can be applied to the subspaces $\text{Der}_{[ \,\cdot\,, Y]}(A \rtimes_\alpha G, \tau)$ and $\text{Der}_{[ \, \cdot \, , X]}(A, \tau)$. To do so, it will become apparent that we need to work with a derivation of the form $u_g^* \cdot D( \alpha_g ( \cdot ))\cdot u_g$, where $D \in \text{Der}_{\text{FR},Y}(A\rtimes_\alpha G,\tau)$.

\begin{lem}\label{Lem:Unitary_map}
Let $G\stackrel{\alpha}{\curvearrowright}(M,\tau)$ be a trace-preserving action of a finite abelian group $G$ on a tracial von Neumann algebra and let $A\subset M$ be a finitely generated unital $*$-subalgebra which is globally invariant under $\alpha$. For each $g \in G$, the map $V_g : \emph{Der}(A\rtimes_\alpha G,\tau) \to \emph{Der}(A\rtimes_\alpha G,\tau)$ defined by 
 $$ D \mapsto u_g^* \cdot D( \alpha_g(\cdot)) \cdot u_g$$
is a unitary with respect to $\langle \cdot, \cdot \rangle_Y$, where $Y$ is a finite self-adjoint generating set of $A\rtimes_\alpha G$ such that $Y$ is scaled under $\alpha$. Furthermore, $V_g$ commutes with the right $((A\rtimes_\alpha G) \otimes (A \rtimes_\alpha G)^\circ)''$-action on $\emph{Der}(A\rtimes_\alpha G,\tau)$ and $ V_g\emph{Der}_{\emph{FR},Y}(A\rtimes_\alpha G,\tau) \subset \emph{Der}_{\emph{FR},Y}(A\rtimes_\alpha G,\tau),$ for all $g \in G$.
\end{lem}
\begin{proof}
First we show that for $g \in G$ and $D \in \text{Der}(A\rtimes_\alpha G)$, $V_g D$ is a derivation. Then for $a,b \in A \rtimes_\alpha G$, one has 
    \begin{align*}
        (V_gD)(ab) &= u_g^* \cdot D(\alpha_g(ab)) \cdot u_g\\
        &= ( u_g^* \otimes u_g^\circ)  D( \alpha_g(a) \alpha_g(b))\\
        &= ( u_g^* \otimes u_g^\circ) [D( \alpha_g(a)) \cdot \alpha_g(b) + \alpha_g(a) \cdot D (\alpha_g(b))]\\
        &= ( u_g^* \otimes u_g^\circ)[ ( u_e \otimes (u_gbu_g^*)^\circ)D( \alpha_g(a))  + (u_gau_g^* \otimes u_e^\circ) D (\alpha_g(b))]\\
        &= ( u_e \otimes b^\circ) ( u_g^* \otimes u_g^\circ)D( \alpha_g(a))  + (a \otimes u_e^\circ)(u_g^* \otimes u_g^\circ) D (\alpha_g(b))]\\
        &=V_gD(a) \cdot b +a \cdot V_gD(b). 
    \end{align*}

Secondly we show that $V_g$ is a unitary for all $g \in G$. Indeed, for $D, D' \in \text{Der}(A\rtimes_\alpha G,\tau)$ and using $\langle \cdot, \cdot \rangle_Y = \langle \cdot , \cdot \rangle_{\alpha_g(Y)}$ in the last equality, it follows that
    \begin{align*}
        \langle V_g(D), D' \rangle_Y &= \sum_{y \in Y} \langle u_g^*\cdot D(\alpha_g(y)) \cdot u_g, D'(y) \rangle\\
        &= \sum_{y \in Y}  \langle (u_g^*\otimes u_g^\circ) D(\alpha_g(y)) , D'(y) \rangle\\
        &= \sum_{y \in Y} \langle  D(\alpha_g(y)) , (u_g\otimes (u_g^*)^\circ) D'(y) \rangle\\
        &= \sum_{y \in Y}  \langle D(\alpha_g(y)) , V_{g^{-1}} D'(\alpha_g(y)) \rangle\\
        &= \langle D, V_{g^{-1}}D' \rangle_{\alpha_g(Y)}\\
        &= \langle D, V_{g^{-1}}D' \rangle_{Y}.
    \end{align*}

Next, for $m \in ((A\rtimes_\alpha G) \otimes (A \rtimes_\alpha G)^\circ)''$, $D \in \text{Der}(A\rtimes_\alpha G,\tau)$ and $b \in A \rtimes_\alpha G$,  one has
 $$[V_g(D \cdot m)](b)= u_g^* \cdot (D \cdot m )(\alpha_g(b)) \cdot u_g =(u_g^* \otimes u_g^\circ) J_{\tau \otimes \tau^\circ} m^* J_{\tau \otimes \tau^\circ} D(\alpha_g(b))$$
and by $J_{\tau \otimes \tau^\circ} m^*J_{\tau \otimes \tau^\circ} \in ((A \rtimes_\alpha G) \otimes (A \rtimes_\alpha G)^\circ)'$, we have
 $$[V_g(D \cdot m)](b)= J_{\tau \otimes \tau^\circ} m^* J_{\tau \otimes \tau^\circ}(u_g^* \otimes u_g^\circ) D(\alpha_g(b))= (V_gD(b))\cdot m.$$
Thus for each $g \in G$, $V_g$ commutes with the right $((A\rtimes_\alpha G) \otimes (A \rtimes_\alpha G)^\circ)''$-action on $\text{Der}(A \rtimes_\alpha G,\tau)$.

Lastly, we show that for each $g \in G$,  $V_g\text{Der}_{\text{FR},Y}(A\rtimes_\alpha G,\tau) \subset \text{Der}_{\text{FR},Y}(A\rtimes_\alpha G,\tau)$. For each $y \in Y$ and $g \in G$, one has $\alpha_{g}(y) = \lambda_{g,y} y$. For $D \in \text{Der}(A\rtimes_\alpha G,\tau)$, we get
 $$ \phi_Y(V_g D)=( ( u_g^* \otimes u_g^\circ) \lambda_{g,y}D(y))_{y \in Y} = ( \delta_{y=y'}( u_g^* \otimes u_g^\circ) \lambda_{g,y})_{y,y'} \phi_Y(D)$$
and it follows that 
 $$\phi_Y V_g \phi^{-1}_Y = ( \delta_{y=y'}( u_g^* \otimes u_g^\circ) \lambda_{g,y})_{y,y'} \in M_{|Y|}( ((A\rtimes_\alpha G) \otimes (A \rtimes_\alpha G)^\circ),\tau \otimes \tau^\circ).$$ 
Notice that
 $$\phi_Y V_g \phi^{-1}_Y (\phi_Y(\text{Der}(A\rtimes_\alpha G,\tau)))\subset \phi_Y(\text{Der}(A\rtimes_\alpha G,\tau))$$ 
for all $g \in G$, which means that the subspace is reducing for $\phi_YV_g\phi_Y^{-1}$ and so it commutes with $[\phi_Y(\text{Der}(A\rtimes_\alpha G,\tau))]$. Since  $$\phi_Y V_g \phi^{-1}_Y (L^2(A \rtimes_\alpha G, \tau) \odot L^2((A \rtimes_\alpha G)^\circ, \tau^\circ))^Y\subset (L^2(A \rtimes_\alpha G, \tau) \odot L^2((A \rtimes_\alpha G)^\circ, \tau^\circ))^Y, $$ 
it follows that $V_g D \in \text{Der}_{\text{FR},Y}(A\rtimes_\alpha G,\tau)$, when $D \in\text{Der}_{\text{FR},Y}(A\rtimes_\alpha G,\tau)$.
\end{proof}

\begin{rem}
Let $Y \subset A\rtimes_\alpha G$ be any finite self-adjoint subset satisfying $A\rtimes_\alpha G = \C\langle Y\rangle$. In particular, $Y$ need not be scaled under the action of $G$. In this case, $V_g$ is no longer a unitary with respect to $\langle \cdot , \cdot \rangle_Y$ for all $g \in G$. But following the proof above, one  still has $\langle V_g(D), D'\rangle_Y = \langle D , V_{g^{-1}}D'\rangle_{\alpha_g(Y)}$.  Since $A \rtimes_\alpha G = \C\langle \alpha_g(Y)\rangle$, for all $y' \in Y$ there exists a polynomial $p_{y'} \in \C\langle T_Y \rangle$ such that $\alpha_{g}(y') = p_{y'}(Y)$. Then we get
$$ (V_gD) (y') =u_g^* \cdot D(\alpha_g(y')) \cdot u_g =( u_g^* \otimes u_g^\circ) D(p_{y'}(Y)) = \sum_{y \in Y} (u_g^*\otimes u_g^\circ)\partial_{y}(p_{y'})(Y) D(y). $$
From the last equality above
$$\phi_Y V_g\phi^{-1}_Y = ( (u_g^*\otimes u_g^\circ)\partial_{y}(p_{y'})(Y) )_{y',y \in Y} \in M_{|Y|} ((A \rtimes_\alpha G)\otimes (A \rtimes_\alpha G)^\circ).$$ 
Notice that 
$$\phi_Y V_g \phi^{-1}_Y (\phi_Y(\text{Der}(A\rtimes_\alpha G,\tau))) \subset \phi_Y(\text{Der}(A\rtimes_\alpha G,\tau))
$$ for all $g \in G$, which means that the subspace is reducing for $\phi_YV_g\phi_Y^{-1}$ and so it commutes with $[\phi_Y(\text{Der}(A\rtimes_\alpha G,\tau))]$. Since $$\phi_Y V_g \phi^{-1}_Y (L^2(A \rtimes_\alpha G, \tau) \odot L^2((A \rtimes_\alpha G)^\circ, \tau^\circ))^Y\subset (L^2(A \rtimes_\alpha G, \tau) \odot L^2((A \rtimes_\alpha G)^\circ, \tau^\circ))^Y, $$ 
it follows that $V_g D \in \text{Der}_{\text{FR},Y}(A\rtimes_\alpha G,\tau)$, when $D \in\text{Der}_{\text{FR},Y}(A\rtimes_\alpha G,\tau)$.
\end{rem}

\begin{lem}\label{Lem:InnDer_cov}
Let $G\stackrel{\alpha}{\curvearrowright}(M,\tau)$ be a trace-preserving action of a finite abelian group $G$ on a tracial von Neumann algebra and let $A\subset M$ be a finitely generated unital $*$-subalgebra which is globally invariant under $\alpha$. For $D \in \emph{Der}_{[ \,\cdot\,, Y]}(\C[G] \subset A\rtimes_\alpha G,\tau)$, there exists $(D_\lambda)_{\lambda \in \Lambda}\subset \emph{InnDer}( \C[G] \subset A\rtimes_\alpha G,\tau)$ that almost weakly approximates $D$, where $Y$ is a finite self-adjoint generating set for $A \rtimes_\alpha G$ such that $Y$ is scaled under $\alpha$.
\end{lem}
\begin{proof}
In light of Lemma \ref{Lem:fin_dim_and_cov}, we need to find a net of inner derivation that almost weakly approximates $D$ such that each inner derivation satisfies the covariant condition, where $D \in  \text{Der}_{[ \,\cdot\,, Y]}(\C[G] \subset A\rtimes_\alpha G,\tau)$. Since $D \in \text{Der}_{[ \,\cdot\,, Y]}(\C[G] \subset A\rtimes_\alpha G,\tau)$, there exists $(D_\lambda)_{\lambda \in \Lambda} \subset \text{InnDer}(A \rtimes_\alpha G,\tau)$ that almost weakly approximates $D$. First, we show that for all $g \in G$, $\{V_g D_\lambda \}_{\lambda \in \Lambda}$ almost weakly approximates $D$. For $g \in G$ and $D'\in \text{Der}_{\text{FR},Y}(A\rtimes_\alpha G,\tau)$, we have by Lemma \ref{Lem:Unitary_map} that $V_{g^{-1}} D' $ is a derivation on $A \rtimes_\alpha G$ and $V_{g^{-1}}D' \in \text{Der}_{\text{FR},Y}(A \rtimes_\alpha G,\tau)$. So using $D \in  \text{Der}(\C[G] \subset A\rtimes_\alpha G,\tau)$ in the second equality,
\begin{align*}
    \langle D-V_g D_\lambda , D'\rangle_Y&= \sum_{y \in Y} \langle D(y)- (V_gD_\lambda)(y) , D'(y)\rangle\\
    &= \sum_{y \in Y} \langle u_g^* \cdot D( \alpha_g(y)) \cdot u_g- u_g^*\cdot D_\lambda (\alpha_g (y))\cdot u_g, D'(y)\rangle\\
    &= \sum_{y \in Y} \langle (D- D_\lambda) ( \alpha_g (y) ),(u_g\otimes (u_g^*)^\circ) D'(y)\rangle\\
    &= \sum_{y \in Y} \langle (D- D_\lambda) ( \alpha_g (y) ),(V_{g^{-1}}D')(\alpha_g(y))\rangle\\
    &=  \langle (D- D_\lambda) ,(V_{g^{-1}}D')\rangle_{\alpha_g(Y)}\\
    &=  \langle (D- D_\lambda) ,(V_{g^{-1}}D')\rangle_{Y}.
\end{align*}
Thus, for all $g\in G$, $\{V_gD_\lambda \}_{\lambda \in \Lambda}$ almost weakly approximates $D$. Hence, we have that $ \left\{\frac{1}{|G|}\sum_{g \in G}V_g D_\lambda \right\}\in \text{InnDer}(\C[G] \subset A\rtimes_\alpha G,\tau)$  almost weakly approximates $D \in \text{Der}_{[ \,\cdot\,, Y]}(\C[G] \subset A\rtimes_\alpha G,\tau)$.
\end{proof}

Notice that the following lemma does not need $G$ to be abelian.
\begin{lem}\label{Lem:Maps_FR}
Let $G\stackrel{\alpha}{\curvearrowright}(M,\tau)$ be a trace-preserving action of a finite group $G$ on a tracial von Neumann algebra and let $A\subset M$ be a finitely generated unital $*$-subalgebra which is globally invariant under $\alpha$ with $X \subset A$ a finite self-adjoint subset satisfying $A =\C\langle X \rangle$ and set $Y = X \cup \{ u_g : g \in G\}$. Then for $h \in G$, we have
\begin{enumerate}
    \item  $d^h \in \emph{Der}_{\emph{FR},Y}(\C[G] \subset A\rtimes_\alpha G,\tau)$, whenever $d \in \emph{Der}_{\emph{FR},X}(A,\tau)$; and
    \item  $D_h \in \emph{Der}_{\emph{FR},X}(A,\tau)$, whenever $D \in \emph{Der}_{\emph{FR},Y}(A \rtimes_\alpha G,\tau)$. 
\end{enumerate}
\end{lem}
\begin{proof}\hfill
\begin{enumerate}
    \item By Theorem \ref{Thm:Bij_Map} and using $\phi_X^{-1}$ and $\phi_Y$, the map $d \mapsto d^h$ can defined on $\phi_X\text{Der}(A,\tau)$ into $\phi_Y\text{Der}(A\rtimes_\alpha G,\tau)$. Since 
     $$J_{\tau\otimes \tau^\circ}(u_e \otimes u_h^\circ)J_{\tau\otimes \tau^\circ}(u_g^*\otimes u_g) (L^2(A,\tau) \odot L^2(A^\circ,\tau^\circ)) \subset L^2(A\rtimes_\alpha G,\tau)\odot L^2((A \rtimes_\alpha G)^\circ,\tau^\circ),$$
    it follows that $d^h \in \text{Der}_{\text{FR},Y}(A\rtimes_\alpha G,\tau)$, whenever $d \in\text{Der}_{\text{FR},X}(A,\tau)$.
    \item By Theorem \ref{Thm:Bij_Map} and using $\phi_Y^{-1}$ and $\phi_X$, the map $D \mapsto D_h$ can defined on $\phi_Y\text{Der}(A,\tau)$ into $\phi_X\text{Der}(A\rtimes_\alpha G,\tau)$. Since 
     $$J_{\tau\otimes \tau^\circ}(u_e \otimes u_h^\circ)J_{\tau\otimes \tau^\circ}p_{e,h} (L^2(A\rtimes_\alpha G,\tau)\odot L^2((A \rtimes_\alpha G)^\circ,\tau^\circ)) \subset (L^2(A,\tau) \odot L^2(A^\circ,\tau^\circ)),$$
     it follows that $d^h \in \text{Der}_{\text{FR},Y]}(A\rtimes_\alpha G,\tau)$, whenever $d \in\text{Der}_{\text{FR},X]}(A,\tau)$. \qedhere
\end{enumerate}
\end{proof}

The issue that arises trying to prove Theorem \ref{Thm:Delta_Dim} is the fact that $\text{Der}_{[\,\cdot\, , X]}(A,\tau)$ depends on the generating set. In the first map $d \to d^h$ of Theorem \ref{Thm:Bij_Map}, the inner products that we have are $\langle \cdot, \cdot \rangle_{\alpha_g(X)}$ for each $g \in G$. However, since $G$ is abelian, we have that $\langle\cdot, \cdot \rangle_{\alpha_g(X)} =\langle\cdot, \cdot \rangle_{X}$ and this further implies that given a net of inner derivation that almost weakly approximates a derivation with respect to $\langle\cdot, \cdot \rangle_{X}$, then the same net almost weakly approximates the same derivation with respect to $\langle\cdot, \cdot \rangle_{\alpha_g(X)}$.

\begin{thm}[{Theorem \ref{Thm:C}}] \label{Thm:Delta_Dim}
Let $G\stackrel{\alpha}{\curvearrowright}(M,\tau)$ be a trace-preserving action of a finite abelian group $G$ on a tracial von Neumann algebra and let $A\subset M$ be a finitely generated unital $*$-subalgebra which is globally invariant under $\alpha$. Then 
 $$\Delta( \C[G] \subset A \rtimes_\alpha G,\tau) = \frac{1}{|G|}\Delta(A,\tau).$$
Furthermore, we have
 $$\Delta(A \rtimes_\alpha G,\tau) -1 = \frac{1}{|G|}(\Delta(A,\tau)-1).$$
\end{thm}
\begin{proof}
Let $X$ be a generating set for $A$ that is scaled under $\alpha$ and set $Y := X \cup \{ u_g : g \in G\}$. First, we show that $d^h \in \text{Der}_{[ \,\cdot\,, Y]}(A\rtimes_\alpha G ,\tau) $, whenever $d \in \text{Der}_{[ \,\cdot\,, X]}(A,\tau)$ and $h \in G$.  For each $d \in \text{Der}_{[ \,\cdot\,, X]}(A,\tau)$, we have by Lemma \ref{Lem:Der_A_to_B} that $d^h \in \text{Der}(\C[G] \subset A\rtimes_\alpha G)$. Since $d \in \text{Der}_{[ \,\cdot\,, X]}(A,\tau)$, there exists $(d_\lambda)_{\lambda \in \Lambda} \subset \text{InnDer}(A,\tau)$ that almost weakly approximates $d$ and we note that $(d_\lambda)^h \in \text{InnDer}(A\rtimes_\alpha G,\tau)$. Let $D' \in \text{Der}_{\text{FR},Y}(A\rtimes_\alpha G,\tau)$ and by Lemma \ref{Lem:Unitary_map} and Lemma \ref{Lem:Maps_FR}, we know that  $V_g D'\in \text{Der}_{\text{FR},Y}(A\rtimes_\alpha G,\tau)$ and $(V_gD')_h\in\text{Der}_{\text{FR},X}(A,\tau)$ for all $g \in G$, respectively. Since $d^h,(d_\lambda)^h \in \text{Der}(\C[G] \subset A \rtimes_\alpha G)$, we have
\begin{align*}
 \langle d^h -(d_\lambda)^h, D'\rangle_Y &=\sum_{x \in X}\langle (d^h-(d_\lambda)^h)(x), D'(x)\rangle\\
 &=\sum_{x \in X}\sum_{g \in G}\langle (d-d_\lambda)(\alpha_g(x)), J_{\tau\otimes \tau^\circ}(u_e \otimes u_h^\circ)J_{\tau\otimes \tau^\circ} (u_g \otimes (u_g^*)^\circ) D'(x)\rangle   
\end{align*}
 and since $(d-d_\lambda)(a) \in L^2(A\otimes A^\circ, \tau \otimes \tau^\circ)$ for $a \in A$ and by Lemma \ref{Lem:rel_of_p_gh}, the computation above becomes
    \begin{align*}
        \langle d^h -(d_\lambda)^h, D'\rangle_Y &=\sum_{x \in X}\sum_{g \in G}\langle (d-d_\lambda)(\alpha_g(x)), p_{e,e}J_{\tau\otimes \tau^\circ}(u_e \otimes u_h^\circ)J_{\tau\otimes \tau^\circ} (V_{g^{-1}}D')(\alpha_g(x))\rangle\\
        &=\sum_{x \in X}\sum_{g \in G}\langle (d-d_\lambda)(\alpha_g(x)), J_{\tau\otimes \tau^\circ}(u_e \otimes u_h^\circ)J_{\tau\otimes \tau^\circ}p_{e,h} (V_{g^{-1}}D')(\alpha_g(x))\rangle,\\
        &=\sum_{x \in X}\sum_{g \in G}\langle  (d-d_\lambda)(\alpha_g(x)), (V_{g^{-1}}D')_h(\alpha_g(x))\rangle\\
        &=\sum_{g \in G} \langle  d-d_\lambda, (V_{g^{-1}}D')_h\rangle_{\alpha_g(X)}\\
        &=\sum_{g \in G} \langle  d-d_\lambda, (V_{g^{-1}}D')_h\rangle_X.
    \end{align*}
Hence $d^h \in \text{Der}_{[ \,\cdot\,, Y]}(\C[G] \subset A \rtimes_\alpha G, \tau)$. Thus the map $d \mapsto d^h$ in Theorem \ref{Thm:Bij_Map} can be restricted to the following subspaces $\text{Der}_{[ \,\cdot\,, Y]}(\C[G] \subset A \rtimes_\alpha G, \tau)$ and $\text{Der}_{[ \, \cdot \, , X]}(A, \tau)$.

Next, we show that for $h \in G$ and $D \in \text{Der}_{[ \,\cdot\,, Y]}(A\rtimes_\alpha G ,\tau)$, we have $D_h \in \text{Der}_{[ \,\cdot\,, X]}(A,\tau)$. Let $D \in \text{Der}_{[ \,\cdot\,, Y]}(\C[G] \subset A\rtimes_\alpha G,\tau)$ and by Lemma \ref{Lem:Der_B_to_A}, we have that $D_h \in \text{Der}(A,\tau)$. By Lemma \ref{Lem:InnDer_cov}, there exists $(D_\lambda)_{\lambda \in \Lambda} \in \text{InnDer}(\C[G] \subset A \rtimes_\alpha G)$ that almost weakly approximates $D$. Let $d' \in \text{Der}_{\text{FR},X}(A,\tau)$ and by Lemma \ref{Lem:Maps_FR}, we have $d^h \in \text{Der}_{\text{FR},Y}(A\rtimes_\alpha G,\tau)$. Then 
    \begin{align*}
        \langle D_h - (D_\lambda')_h , d' \rangle_X &= \sum_{x \in X} \langle J_{\tau \otimes \tau^\circ}(u_e \otimes u_h^\circ) J_{\tau \otimes \tau^\circ} p_{e,h} (D - D_\lambda')(x) , d'(x) \rangle\\
        &= \sum_{x \in X} \langle (D - D_\lambda')(x) ,p_{e,h} J_{\tau \otimes \tau^\circ}(u_e \otimes (u_h^*)^\circ) J_{\tau \otimes \tau^\circ}d'(x) \rangle,
        \end{align*}
and since $d'(a) \in L^2(A ,\tau) \odot L^2(A^\circ,\tau^\circ)$ for $a \in A$, we can further compute
\begin{align*}
        \langle D_h - (D_\lambda')_h , d' \rangle_X&= \sum_{x \in X} \langle (D - D_\lambda')(x) , J_{\tau \otimes \tau^\circ}(u_e \otimes (u_h^*)^\circ) J_{\tau \otimes \tau^\circ}d'(x) \rangle\\
        &= \frac{1}{|G|}\sum_{g \in G} \sum_{x \in X} \langle (D - D_\lambda')(\alpha_g(x)) , J_{\tau \otimes \tau^\circ}(u_e \otimes (u_h^*)^\circ) J_{\tau \otimes \tau^\circ}d'(\alpha_g(x)) \rangle\\
        &= \frac{1}{|G|}\sum_{g \in G} \sum_{x \in X} \langle (D - D_\lambda')(x) , J_{\tau \otimes \tau^\circ}(u_e \otimes (u_h^*)^\circ) J_{\tau \otimes \tau^\circ}( u_g^* \otimes u_g^\circ)d'(\alpha_g(x)) \rangle\\
        &=\frac{1}{|G|}\sum_{x \in X} \langle (D - D_\lambda')(x) , (d')^h(x) \rangle.
    \end{align*}
Since $D, D'_\lambda, (d')^h \in  \text{Der}(\C[G] \subset A\rtimes_\alpha G,\tau)$, the computation becomes
$$ \langle D_h - (D_\lambda')_h , d' \rangle_X =\frac{1}{|G|} \sum_{y \in Y} \langle (D - D_\lambda')(y) , (d')^h(y) \rangle=\frac{1}{|G|}\langle D - D_\lambda' , (d')^h \rangle_Y.$$
Hence, $D_h \in \text{Der}_{[\cdot ,X]}(A,\tau)$. Thus the map $D \mapsto D_h$ in Theorem 2.5 can be restricted to the following subspaces $\text{Der}_{[ \,\cdot\,, Y]}(\C[G] \subset A \rtimes_\alpha G, \tau)$ and $\text{Der}_{[ \, \cdot \, , X]}(A, \tau)$.
 
By applying the third map in Theorem \ref{Thm:Bij_Map}, one has the right $(A \otimes A^\circ)''$-module isomorphism
    $$ \text{Der}_{[ \,\cdot\,, Y]}(\C[G] \subset A \rtimes_\alpha G,\tau) = \bigoplus_{g \in G} (\text{Der}_{[ \,\cdot\,, X]}(A,\tau))_{1 \otimes \alpha_g},$$
where the direct sum is with respect to the $\langle \cdot, \cdot \rangle_Y$, where $Y = X \cup \{u_g : g \in G\}$.

Lastly, by Lemma \ref{Lem:Decomp} and $\text{Der}_{[\cdot , \{u_g : g \in G\}]}(\C[G],\tau)= \overline{\text{InnDer}(\C[G],\tau)}$, we can apply Corollary \ref{Cor:Bij_Map_Formula} to the following subspaces $\text{Der}_{[ \,\cdot\,, Y] } (A\rtimes_\alpha G,\tau) \subset \text{Der}(A\rtimes_\alpha G,\tau)$ and $\text{Der}_{[ \,\cdot\,, X]}(A,\tau) \subset \text{Der}(A,\tau)$ to get
    \begin{align*}
        \Delta( \C[G] \subset A \rtimes_\alpha G, \tau) &= \dim\text{Der}_{[ \,\cdot\,, Y]}(\C[G] \subset A \rtimes_\alpha G,\tau)_{((A\rtimes_\alpha G) \otimes (A\rtimes_\alpha G)^\circ)''} \\
        &= \frac{1}{|G|}  \dim \text{Der}_{[ \, \cdot \, , X]}(A,\tau)_{(A \otimes A^\circ)''}\\
        &= \frac{1}{|G|} \Delta(A,\tau),
    \end{align*} 
and
    \begin{align*}
        \Delta(A \rtimes_\alpha G, \tau)-1 &= \dim\text{Der}_{[\,\cdot\, , Y]}(A \rtimes_\alpha G,\tau)_{((A\rtimes_\alpha G) \otimes (A\rtimes_\alpha G)^\circ)''} -1\\
        &= \frac{1}{|G|} ( \dim \text{Der}_{[ \, \cdot \, , X]}(A,\tau)_{(A \otimes A^\circ)''} -1)\\
        &= \frac{1}{|G|} (\Delta(A,\tau)-1).\qedhere
    \end{align*} 
\end{proof}

The following corollary follows by a similar proof to Corollary \ref{Cor:Subgroup}.
\begin{cor}
Let $G\stackrel{\alpha}{\curvearrowright}(M,\tau)$ be a trace-preserving action of a finite abelian group $G$ on a tracial von Neumann algebra and let $A\subset M$ be a finitely generated unital $*$-subalgebra which is globally invariant under $\alpha$. If $H \subset G$ is a finite subgroup of $G$, then 
 $$\Delta( A \rtimes_\alpha G,\tau) -1= \frac{1}{[G: H]} (\Delta ( A\rtimes_\alpha H,\tau)-1).$$
\end{cor}

The following is an estimate for the free entropy dimension $\delta_0$ when we consider the crossed product of a von Neumann algebra with a finite abelian group, this uses a known inequality, $\delta_0 \leq \Delta$ (see \cite[Corollary 4.6]{CS05}).
\begin{cor}\label{Cor:Dim_Bounds}
Let $G\stackrel{\alpha}{\curvearrowright}(M,\tau)$ be a trace-preserving action of a finite abelian group $G$ on a tracial von Neumann algebra and let $A\subset M$ be a finitely generated unital $*$-subalgebra which is globally invariant under $\alpha$. Then for any generating set $Y$ of $A \rtimes_\alpha G$, we have
\begin{equation}\label{Eq:Micro_FED}
    \delta_0(Y)  \leq  \frac{1}{|G|} (\Delta(A,\tau)-1) +1.
\end{equation} 
In particular, if $A$ is a weak operator topology dense subset of $M$ with $X$ is a finite self-adjoint generating subset of $A$ then for any generating set $Y$ of $A\rtimes_\alpha G$, we have
$$  \delta_0(Y) \leq  \frac{1}{|G|} (|X| -1) +2.$$
\end{cor}
\begin{proof}
Using \cite[Corollary 4.6]{CS05} and Theorem \ref{Thm:Delta_Dim}, we have
$$\delta_0(Y) \leq \Delta(A \rtimes_\alpha G,\tau)= \frac{1}{|G|}(\Delta(A,\tau)-1)+1.$$

Lastly, let $A= \C \langle X \rangle$ be weak operator topology dense in $M$. Then, since $\Delta$ is a $*$-algebra invariant, we further obtain
\begin{equation*}
    \delta_0(Y) = \frac{1}{|G|}(\Delta(A,\tau)-1)+1\leq \frac{1}{|G|}(|X| -1) +1. \qedhere
\end{equation*}
\end{proof}

Similar to Example \ref{Ex:Fin_Gen_gp_Sigma}, the following example we consider $G$ to be a countable abelian group.

\begin{ex}\label{Ex:Fin_Gen_gp_Delta}
Let $G$, $G_n$, $L(G)$, $M$, $\alpha$, $A$, and $Y$ be as in Example \ref{Ex:Fin_Gen_gp_Sigma}. We use the same proof but using \cite[Theorem 3.3]{CS05} and \cite[Corollary 3.5]{CS05} instead of \cite[Corollary 2.13]{CN21} and \cite[Corollary 3.3]{CN22}, respectively, one has
$$ \Delta( (A\rtimes_\alpha G_n) \vee \C\langle Y\rangle, \tau ) \leq 1 +\frac{1}{|G_n|} \Delta(A,\tau).$$
Hence, for all $\varepsilon>0$, $A \rtimes_\alpha G$ admits a dense $*$-subalgebra $B$ with $\Delta(B,\tau) \leq 1 + \varepsilon$. By Corollary \ref{Cor:Dim_Bounds}, it follows that for all $\varepsilon>0$, $M \rtimes_\alpha G$ admits a generating set $Y$ such that $\delta_0(Y) \leq 1 + \varepsilon. \hfill \blacksquare$
\end{ex}

In \cite[Corollary 5]{Shl22} Shlyakhtenko showed that there exists a particular generating set $Y$ of $A \rtimes_\alpha G$ such that $\delta_0(Y) \leq |G|^{-1}(2 |X| +2)+1$. Corollary \ref{Cor:Dim_Bounds} shows that any generating set of $A \rtimes_\alpha G$ can be used to obtain the sharper bound in equality (\ref{Eq:Micro_FED})

Although, it is still unknown if $\delta^*, \delta^\star,\delta_0$ are $*$-algebra invariants; when $\Delta$ and $\dim \text{Der}_c$ agree, one has that $\delta^*$, $\delta^\star$ and $\delta_0$ are indeed $*$-algebra invaraints by \cite[Theorem 2]{Shl09} and \cite[Lemma 4.1, Theorem 4.4]{CS05}. The following corollary extends the class where $\delta^*$, $\delta^\star$ and $\delta_0$ agree and are $*$-algebra invaraints.
\begin{cor}
Let $G\stackrel{\alpha}{\curvearrowright}(M,\tau)$ be a trace-preserving action of a finite abelian group $G$ on a tracial von Neumann algebra and let $A\subset M$ be a finitely generated unital $*$-subalgebra which is globally invariant under $\alpha$. If we have $\dim \emph{Der}_c(A,\tau)_{(A \otimes A^\circ)''} = \Delta(A,\tau)$ and $A''$ can be embedded in the ultrapower of the hyperfinite $\emph{II}_1$ factor, then for any generating set $Y$ of $A \rtimes_\alpha G$, one has
$$\dim \overline{\emph{Der}_c(A,\tau)}_{(A \otimes A^\circ)''} = \delta_ 0(Y) = \delta^*(Y)= \delta^\star(Y) = \Delta(A \rtimes_\alpha G,\tau).$$
If we instead have $\sigma(A,\tau) = \Delta(A,\tau)$, then for any generating set $Y$ of $A \rtimes_\alpha G$, one has
$$\sigma(A \rtimes_\alpha G,\tau) = \delta^*(Y)= \delta^\star(Y) = \Delta(A \rtimes_\alpha G,\tau).$$
\end{cor}
\begin{proof}
We only show the first set of equalities, since the other uses a similar proof but with \cite[Corollary 4.4]{CN21} and Theorem \ref{Thm:Free_Stein_Dim}. By \cite[Theorem 2]{Shl09} and \cite[Lemma 4.1, Theorem 4.4]{CS05}, we have 
$$\dim\overline{ \text{Der}_c(A,\tau)}_{(A \otimes A^\circ)''} \leq \delta_0(X) \leq \delta^*(X) \leq \delta^\star(X) \leq \Delta(A,\tau).$$
Since $\dim \overline{\text{Der}_c(A,\tau)}_{(A \otimes A^\circ)''}= \Delta(A,\tau)$ it follows that 
$$\dim \overline{\text{Der}_c(A,\tau)}_{(A \otimes A^\circ)''} = \delta_0(x) =\delta^*(X) = \delta^\star(X) = \Delta(A,\tau).$$
It follows from Theorem \ref{Thm:Anot_subs} and Theorem \ref{Thm:Delta_Dim} that
\begin{align*}
    \dim \overline{\text{Der}_c(A \rtimes_\alpha G, \tau)}_{((A\rtimes_\alpha G) \otimes (A\rtimes_\alpha G)^\circ)''} &= \frac{1}{|G|}( \dim \overline{\text{Der}_c(A,\tau)}_{(A \otimes A^\circ)''} -1) +1\\
    &=\frac{1}{|G|}( \Delta(A,\tau) -1) +1\\
    &= \Delta(A \rtimes_\alpha G, \tau).
\end{align*} 
Thus, we have $\dim \overline{\text{Der}_c(A \rtimes_\alpha G, \tau)}_{((A\rtimes_\alpha G) \otimes (A\rtimes_\alpha G)^\circ)''} =\delta_0(Y) = \delta^*(Y)= \delta^\star(Y) = \Delta(A \rtimes_\alpha G,\tau)$.
\end{proof}


\bibliographystyle{amsalpha}
\bibliography{references}

\end{document}